\documentclass[a4paper]{amsart}
\usepackage{amscd,amsfonts,amssymb}
\usepackage{tikz}
\usepackage{pgfplots}
\usetikzlibrary{automata,positioning,calc,trees}
\usetikzlibrary{intersections,pgfplots.fillbetween}
\usepackage{graphicx,tikz}
\usepackage{pgf,tikz,pgfplots}
\usepackage{mathrsfs}
\usetikzlibrary{arrows}
\usepackage{enumerate}
\usepackage[shortlabels]{enumitem}
\usepackage{mathrsfs}
\usepackage{amssymb,amsmath,amsthm,color}
\usepackage{caption,subcaption}
\usepackage[alphabetic,bibtex-style]{amsrefs}
\usepackage{hyperref}
\usepackage{url}
\usepackage{setspace}
\usepackage{float}
\BibSpec{article}{%
	+{}  {\PrintAuthors}                {author}
	+{,} { \textit}                     {title}
	+{.} { }                            {part}
	+{:} { \textit}                     {subtitle}
	+{,} { \PrintContributions}         {contribution}
	+{.} { \PrintPartials}              {partial}
	+{,} { }                            {journal}
	+{}  { \textbf}                     {volume}
	+{}  { \PrintDatePV}                {date}
	+{,} { \issuetext}                  {number}
	+{,} { \eprintpages}                {pages}
	+{,} { }                            {status}
	+{,} { \url}                        {url}    
	+{,} { \PrintDOI}                   {doi}
	+{,} { available at \eprint}        {eprint}
	+{}  { \parenthesize}               {language}
	+{}  { \PrintTranslation}           {translation}
	+{;} { \PrintReprint}               {reprint}
	+{.} { }                            {note}
	+{.} {}                             {transition}
}
\BibSpec{book}{%
	+{}  {\PrintAuthors}                {author}
	+{,} { \textit}                     {title}
	+{.} { }                            {part}
	+{:} { \textit}                     {subtitle}
	+{,} { \PrintContributions}         {contribution}
	+{.} { \PrintPartials}              {partial}
	+{,} { }                            {series}
	+{}  { \textbf}                     {volume}
	+{}  { \PrintDatePV}                {date}
	+{,} { \issuetext}                  {number}
	+{,} { \eprintpages}                {pages}
	+{,} { }                            {status}
	+{,} { \url}                        {url}    
	+{,} { available at \eprint}        {eprint}
	+{}  { \parenthesize}               {language}
	+{}  { \PrintTranslation}           {translation}
	+{;} { \PrintReprint}               {reprint}
	+{.} { }                            {note}
	+{.} {}                             {transition}
	+{,} { \PrintDOI}                   {doi}
}

\newtheorem{thm}{Theorem}[section]

\newtheorem{theorem}[thm]{Theorem}

\newtheorem{remark}[thm]{Remark}
\newcounter{old}

\newtheorem{lemma}[thm]{Lemma}
\newtheorem{prop}[thm]{Proposition}
\theoremstyle{definition}
\newtheorem{defn}[thm]{Definition}
\theoremstyle{remark}
\newtheorem{rem}[thm]{Remark}
\numberwithin{equation}{section}
	\newcommand{\ee}{\end{equation}}

\def\sqw{\hbox{\rlap{\leavevmode\raise.3ex\hbox{$\sqcap$}}$%
		\sqcup$}}
\def\findem{\ifmmode\sqw\else{\ifhmode\unskip\fi\nobreak\hfil
		\penalty50\hskip1em\null\nobreak\hfil\sqw
		\parfillskip=0pt\finalhyphendemerits=0\endgraf}\fi}

\newcommand{\R}{{\mathbb {R}}}

\newcommand{\N}{{\mathbb N}}
\newcommand{\Z}{{\mathbb Z}}

\newcommand{\s}{\mathbf S}

\newcommand{\supp}{\operatorname{supp}}
\begin{document}
\baselineskip16pt
	
\title[]{$L^{p}$-estimates for uncentered spherical averages and lacunary 
maximal functions}

\author[Bhojak]{Ankit Bhojak}
\address{Ankit Bhojak\\
	Department of Mathematics\\
	Indian Institute of Science Education and Research Bhopal\\
	Bhopal-462066, India.}
\email{ankitb@iiserb.ac.in}

\author[Choudhary]{Surjeet Singh Choudhary}
\address{Surjeet Singh Choudhary\\
	Department of Mathematics\\
	Indian Institute of Science Education and Research Bhopal\\
	Bhopal-462066, India.}
\email{surjeet19@iiserb.ac.in}

\author[Shrivastava]{Saurabh Shrivastava}
\address{Saurabh Shrivastava\\
	Department of Mathematics\\
	Indian Institute of Science Education and Research Bhopal\\
	Bhopal-462066, India.}
\email{saurabhk@iiserb.ac.in}

\author[Shuin]{Kalachand Shuin}
\address{Kalachand Shuin\\
	Department of Mathematics\\
	Indian Institute of Science\\
	Bengaluru-560012, India.}
\email{kalachands@iisc.ac.in}	
	
	
\date{\today}
	
\begin{abstract}
	The primary goal of this paper is to introduce bilinear analogues of 
	uncentered spherical averages, Nikodym averages associated with spheres and the 
	associated bilinear maximal functions.
	We obtain $L^p$-estimates for uncentered bilinear maximal functions for dimensions $d\geq2$. Moreover, we also discuss the one-dimensional case. 
	In the process of developing these results, we also establish new and interesting results in the linear case. 
	In particular, we will prove $L^p$-improving properties for single 
	scale averaging operators and $L^p$-estimates for lacunary maximal functions in this context.

\end{abstract}
\maketitle
\tableofcontents
\section{Introduction} \label{section1}
The classical Kakeya needle problem concerns the region with minimum area 
needed to rotate a needle of length one by $2\pi$ radians in the plane. 
The problem was resolved by a construction of Besicovitch, where he showed 
that the needle can be rotated in a region of arbitrary small area. 
Cunningham \cite{Cunningham} proposed a similar problem for circular 
arcs; more precisely, he asked if it is possible to move a circular 
arc from one point to another in a region with arbitrary small area. 
This question was answered in the positive by H\'era and 
Laczkovich \cite{KakeyaForCircularArcs} for sufficiently 
small arcs. Chang and Cs\"ornyei \cite{ChangCsornyei} showed 
the similar property holds for the set obtained by removing 
arbitrary small neighbourhoods of diametrically opposite points 
from a circle. These problems are closely related to certain maximal 
averaging operators \cite{BourgainCircular, MattilaFourierbook, Wolff}.

Recently Chang, Dosidis and Kim \cite{NikodymSetsAndMaximalFunctionsAssociatedWithSpheres} 
considered the Kakeya problem for the $(d-1)$-dimensional sphere 
and obtained $L^p$-bounds for the uncentered maximal averaging operators 
associated to the spheres. Motivated by these developments, we study 
off-diagonal $L^p$-estimates for uncentered single scale averages 
and $L^p$-estimates for the corresponding uncentered lacunary maximal averages. 
In view of the extensive progress made in the theory of bilinear spherical maximal 
functions \cite{GGIPS,GHH,ImprovedBoundsForTheBilinearSphericalMaximalOperators, 
MaximalEstimatesForTheBilinearSphericalAveragesAndTheBilinearBochnerRieszOperators,
IPS,ChristZhou,DosidisRamos,Borges,GGPP,RSS,ShrivastavaShuin}, we extend the notion of uncentered 
spherical averages and Nikodym type maximal averages to bilinear setting and 
establish $L^p$-estimates for them in all dimensions. Since the introductory material 
on uncentered spherical averages and Nikodym averages requires different notation, we dedicate separate 
sections to describe the 
results on each topic mentioned above along with the necessary preliminary background.  

\subsection*{Organization of the paper} The rest of the paper is organized 
as follows. 
\begin{itemize}
	\item In Section~\ref{Sec:uncentered}, we discuss some known results for uncentered spherical maximal functions. Later, we describe our results on $L^p$-improving properties and lacunary maximal functions in this direction. These results are proved in Sections~\ref{Sec:prooflacuncentered} and \ref{Sec:ProofLp}.
	\item Section~\ref{Sec:Nikodym} is devoted to the discussion of Nikodym maximal function associated with spheres. In this section, we state our result on the lacunary Nikodym maximal function and prove it in~Section~\ref{Sec:proofNikodym}.  
	\item The bilinear analogues of the results discussed in Sections~\ref{Sec:uncentered} and \ref{Sec:Nikodym} are described in Section~\ref{Sec:bilinear}. Section \ref{Sec:prooflacbilinearuncentered} is devoted to prove the main result of the paper on the bilinear lacunary maximal operator $\mathcal N^T_{lac}$. Proof of bilinear Nikodym maximal function is discussed in Section~\ref{Sec:proofNikodym}. 
\end{itemize}

\section{Uncentered spherical averages and maximal functions}\label{Sec:uncentered}
Let $f\in\mathcal{S}(\mathbb{R}^{d}), d\geq 2$. Given $u\in\mathbb{R}^{d},$ the 
uncentered spherical  average of $f$ is defined by  
$$A^{u}_{t}f(x)=\int_{\mathbb{S}^{d-1}}f(x+t(u+y))~d\sigma(y),~t>0, $$
where $d\sigma$ is the normalized surface measure on the sphere $\mathbb{S}^{d-1}$. 
For $t=1$, we use the notation $A^{u}_{1}=A^{u}.$ Note that if 
$u=0, A^{0}_{t}f$ is the standard spherical average. The spherical 
averages are well-studied in the literature in various contexts. 
For example, it is well-known that the spherical average $A^{0}_{t}f$ 
appears as a solution to the wave equation. Littman \cite{Littman} 
investigated  $L^{p}\rightarrow L^{q}$-estimates of the operator 
$A^{0}$ for $d\geq2$. He proved the following $L^p$-improving estimates 
for spherical averages:
\begin{eqnarray}\label{Littman}
	\Vert A^{0}f\Vert_{L^{q}(\mathbb{R}^{d})}\lesssim \Vert f\Vert_{L^{p}(\mathbb{R}^{d})}
\end{eqnarray} 
for $(\frac{1}{p},\frac{1}{q})$ belonging to the triangle with vertices  $\{(0,0),(1,1),(\frac{d}{d+1},\frac{1}{d+1})\}$. 
Here, the notation $A\lesssim B$ means that there exists a constant $C>0$ such that $A\leq CB$. We 
also refer the reader to  ~\cite{Strichartz} and \cite{SparseBoundsForSphericalMaximalFunction} 
for $L^p$-improving properties of spherical averages. The corresponding $L^p$-improving estimates 
for $A^{0}_{t}$ follows by scaling. 

Recently,  Chang, Dosidis and Kim~\cite{NikodymSetsAndMaximalFunctionsAssociatedWithSpheres} 
introduced maximal functions associated with uncentered spherical averages and proved interesting 
results. Given a compact set  $T\subset\mathbb{R}^{d},$ they considered the maximal function
$$A^{T}f(x):=\sup_{u\in T}|A^{u}f(x)|$$
and the full maximal function associated with uncentered spherical averages 
$$S^{T}f(x):=\sup_{u\in T}\sup_{t>0}|A^{u}_{t}f(x)|.$$
Observe that if $T=\{0\}$, the maximal operator $S^T$ corresponds to the classical spherical 
maximal function, denoted by $M_{full}$. The spherical maximal operator $M_{full}$ has been 
studied extensively. Stein \cite{SteinMaximalFunctionsSpherical} and Bourgain \cite{BourgainCircular} 
proved sharp $L^p$-estimates for the spherical maximal function for $d\geq3$ and $d=2$ respectively. 
The $L^p$-estimates for $M_{full}$ holds in the range $p>\frac{d}{d-1},$ which is known to be optimal. 
If we restrict the supremum in the definition of $M_{full}$ to lacunary sequences, the resulting 
maximal function is commonly referred to as the lacunary spherical maximal function and is denoted 
by $M_{lac}$. The range of $p$ for which $M_{lac}$ satisfies $L^p$-estimates extends to $p>1.$ 
We refer the reader to~\cite{SparseBoundsForSphericalMaximalFunction} for recent developments 
on operators $M_{lac}$ and $M_{full}$. 

\sloppy Chang, Dosidis and Kim~\cite{NikodymSetsAndMaximalFunctionsAssociatedWithSpheres} showed 
that in order to get non-trivial $L^p$-estimates for operators $A^T$ and $S^T$, we require 
suitable assumptions on `size' of the underlying set $T$. For example, if we consider 
$T=\mathbb{S}^{d-1},$ then the operator $A^{\mathbb{S}^{d-1}}$ is not bounded in 
$L^{p}(\R^n)$ for any $p<\infty$. This assertion follows from the existence of Nikodym sets 
for spheres,~see\cite{NikodymSetsAndMaximalFunctionsAssociatedWithSpheres} and references 
therein for precise details. Let us recall the notion of Nikodym sets in the context of lines and spheres. 
\begin{defn}\cite{NikodymSetsAndMaximalFunctionsAssociatedWithSpheres, MattilaGeometryOfSetsMeasureEucli}
	\begin{enumerate}
		\item A Nikodym set for lines is a set $A\subset\mathbb{R}^{d}$ of 
		Lebesgue measure zero such that for every $x\in\mathbb{R}^d,$ there is a line $\tau$ which 
		passes through $x$ and $A\cap \tau$ contains a unit line segment. 
		\item Similarly, a Nikodym set for spheres (resp., unit spheres) is a set $T\subset \R^d$ 
		of Lebesgue measure zero such that for every $y$ in a set of positive Lebesgue measure, there 
		exists a sphere (resp., unit sphere) $S$ containing $y$ such that $A\cap S$ has positive 
		$(d-1)$-dimensional Hausdorff measure. 
	\end{enumerate}
\end{defn}
Nikodym sets are closely related with Kakeya sets and Besicovitch sets. The interested reader 
is referred to Mattila~\cite{MattilaGeometryOfSetsMeasureEucli} for more details about these sets. 
Note that the Nikodym sets for spheres in $\R^d$ are small with respect to the $d$- dimensional 
Lebesgue measure. However, their Hausdorff dimension must 
be $d$, see~\cite{NikodymSetsAndMaximalFunctionsAssociatedWithSpheres} for details. 

In \cite{NikodymSetsAndMaximalFunctionsAssociatedWithSpheres}, Chang, Dosidis and Kim studied 
$L^{p}$-boundedness of maximal functions $A^{T}$ and $S^T$ under suitable conditions on the 
size of $T$ in terms of Minkowski content. A compact set $T\subset\mathbb{R}^{d}$ is said to 
have finite $s$-dimensional upper Minkowski content if for all $\delta\in (0,\frac{1}{2})$, 
$$N(T,\delta)\lesssim \delta^{-s},$$
where $N(T,\delta)$ denotes  the minimal number of balls of radius $\delta$ needed to cover 
$T.$ They proved the following result.
\begin{theorem}[\cite{NikodymSetsAndMaximalFunctionsAssociatedWithSpheres},Theorem 1.7]\label{NikodymSetsAndMaximalFunctionsAssociatedWithSpheresNT}
	Let $d\geq2$ and $0\leq s<d-1$. Let $T\subset \mathbb{R}^{d}$ be a compact set with 
	finite $s$-dimensional upper Minkowski content. Then $A^{T}$ is bounded from 
	$L^{p}(\mathbb{R}^{d})$ to itself in each of the following cases.  
	\begin{enumerate}[label=(\roman*)]
		\item when $d=2$ and $p>1+s$,
		\item when $d=3$ and $$p>1+\min\Big(\frac{s}{2},\frac{1}{3-s},\frac{5-2s}{9-4s}\Big),$$
		\item when $d\geq4$ and $$p>1+\min\Big(\frac{s}{d-1},\frac{1}{d-s},\frac{d-s}{3(d-s)-2}\Big).$$
	\end{enumerate}
\end{theorem}
The following result is known for the operator $S^T$. 
\begin{theorem}[\cite{NikodymSetsAndMaximalFunctionsAssociatedWithSpheres}, Theorem $1.10$]\label{NikodymSetsAndMaximalFunctionsAssociatedWithSpheresfull}
	Let $T$ be a compact subset of $\mathbb{R}^{d}$ with upper Minkowski content $0<s<d-1$, 
	\begin{enumerate}[label=(\roman*)]
		\item when $d=2$ and $0<s<1$, $S^{T}$ maps  $L^{p}(\mathbb{R}^{d})$ to $L^{p}(\mathbb{R}^{d})$ 
		for $p>p_{2,s}=2+\min\{1,\max(s,\frac{4s-2}{2-s})\}$,
		\item when $d\geq3$ and $0<s<d-1$,  $S^{T}$ maps  $L^{p}(\mathbb{R}^{d})$ to $L^{p}(\mathbb{R}^{d})$ 
		for $p>p_{d,s}=1+[d-1-s+\max(0,\min(1,(2s-d+3)/4))]^{-1}$.
	\end{enumerate}
\end{theorem}
In the same paper, the authors have obtained certain necessary conditions on $p$ for the $L^{p}$-boundedness 
of $S^{T}$ to hold. These necessary conditions are sharp in some cases.

The primary goal of this paper is to  
introduce the bilinear analogues of maximal functions discussed as above and 
establish 
the desired 
$L^p$-estimates for them. This constitutes the main body of the paper. 
However, we notice that some of the 
important questions about their linear counterparts are not yet completely 
addressed. 
Therefore, we take up the study of these questions in the second part of 
the paper. In particular, 
we obtain $L^p$-improving properties of spherical averages and $L^p$-estimates 
for lacunary maximal 
functions in the linear case. These results complement the work done 
in~\cite{NikodymSetsAndMaximalFunctionsAssociatedWithSpheres}. 

We describe the main results of this paper in linear and bilinear setting 
separately in the following two sections. 
\section{Main results in the linear case}\label{Sec:linearcase}
\subsection{Results for lacunary maximal function $N^{T}_{lac}$}\label{Sec:lacuncentered}
Motivated by the discussion above, we consider the lacunary analogue of maximal 
averages and show that the 
range of $p$ in Theorem \ref{NikodymSetsAndMaximalFunctionsAssociatedWithSpheresfull} could be improved 
significantly for the lacunary uncentered spherical maximal function. This is 
defined as follows  
$$N^{T}_{lac}f(x):=\sup_{k\in\mathbb{Z}}\Big|A^{T}_{2^k}f(x)\Big|=\sup_{u\in T,k\in\mathbb{Z}}\Big|\int_{\mathbb{S}^{d-1}}f(x+2^{k}(u+y))d\sigma(y)\Big|.$$
Once again, observe that if $T=\{0\},$ the operator $N^{0}_{lac}$ coincides with 
the classical 
lacunary spherical maximal function $M_{lac}.$ We would like to refer the reader 
to  
Calder\'{o}n \cite{CalderonLacunarySphericalMeans}, Coifman-Weiss \cite{CoifmanWeissBookReview}, 
Lacey~\cite{SparseBoundsForSphericalMaximalFunction} for $L^p$-estimates of 
$N^{0}_{lac}=M_{lac}$ 
for $1<p\leq\infty$. Also, see Seeger, Tao and Wright~\cite{EndpointMappingPropertiesOfSphericalMaximalOperators} 
for weak-type  $L\log\log L$-estimate for the operator $M_{lac}.$
We have the following result for the lacunary maximal function $N^{T}_{lac}$. 
\begin{theorem}\label{NT}
	Let $d\geq 2$ and $0\leq s<d-1$. Suppose that $T\subset \R^d$ is a compact 
	set with finite $s$-dimensional 
	upper Minkowski content. Then $N^{T}_{lac}$ is bounded on $L^p(\R^d)$ for 
	each of the following cases. 
	\begin{enumerate}[label=(\roman*)]
		\item $d=2$ and  $p> 1 + s,$
		\item $d=3$ and $p> 1+ \min\left(\frac{s}{2}, \frac{1}{3-s}, \frac{5-2s}{9-4s}\right),$
		\item $ d\geq 4$ and
		$p> 1+ \min\left(\frac{s}{d-1},\frac{1}{d-s}, \frac{d-s}{3(d-s)-2}\right).$
		\item Moreover, $N_{lac}^T$ is of restricted weak-type $(p,p)$ at the endpoint $p=1+\min\left(\frac{s}{d-1},\frac{1}{d-s}\right)$.
	\end{enumerate}
	
\end{theorem}
\begin{remark}
	Observe that the range of $p$ in Theorems \ref{NikodymSetsAndMaximalFunctionsAssociatedWithSpheresNT} 
	and \ref{NT} are the same.
\end{remark}
\begin{remark}
	Theorem \ref{NT} implies the pointwise a.e. convergence of the averaging 
	operator 
	$A^{T}_{2^{k}}f$, see Section $7$ in \cite{OnPointwisea.e.ConvergenceOfMultilinearOperators} 
	for the necessary details.
\end{remark}

\subsection{$L^p-$improving estimates}

We extend the $L^p$-estimates of $A^T$ to the off-diagonal range. These are 
referred to as $L^{p}$-improving estimates for the operator $A^{T}$. This 
completes the picture of $L^{p}\to L^{q}$-estimates for the operator $A^{T}$ 
with $p\leq q.$ We would like to point out that $L^{p}$-improving estimates 
for averaging operators play a crucial role in proving sparse domination of 
the corresponding maximal operators. For example, 
see \cite{SparseBoundsForSphericalMaximalFunction} for this connection 
in the context of classical spherical maximal functions. We also refer 
to \cite{RSS,Palssonbilinearsparse,BFOPZ} for sparse domination of bilinear 
spherical maximal functions. However, we do not pursue the direction of sparse 
domination of maximal operators in this paper.

We need to introduce some notation in order to state the $L^p$-improving results. 

Let $\Delta$ denote the closed triangle with vertices $A,H$ and $E$, where  $A=(0,0)$, $E=\left(\frac{d-s}{d-s+1},\frac{1}{d-s+1}\right)$ and 
\begin{itemize}
	\item $H=\Big(\frac{1}{1+s}, \frac{1}{1+s}\Big),$ if $d=2$, 
	\item $H=\left(\left(1+\min\{\frac{s}{2},\frac{1}{3-s},\frac{5-2s}{9-4s}\}\right)^{-1},\left(1+\min\{\frac{s}{2},\frac{1}{3-s},\frac{5-2s}{9-4s}\}\right)^{-1}\right),$ if $d=3$,
	\item $H=\left(\left(1+\min\{\frac{s}{d-1},\frac{1}{d-s},\frac{d-s}{3(d-s)-2}\}\right)^{-1},\left(1+\min\{\frac{s}{d-1},\frac{1}{d-s},\frac{d-s}{3(d-s)-2}\}\right)^{-1}\right),$ if $d\geq4$.
\end{itemize}
Observe that, if $d\geq4$ and $d-2< s<d-1,$ then 
$\min\{\frac{s}{d-1},\frac{1}{d-s},\frac{d-s}{3(d-s)-2}\}=\frac{d-s}{3(d-s)-2}.$ Similarly, if $d=3$ and $\frac{3}{2}< s<2$, then  
$\min\{\frac{s}{2},\frac{1}{3-s},\frac{5-2s}{9-4s}\}=\frac{5-2s}{9-4s}.$
With the notation as above, we have the following $L^p$-improving properties of $A^T.$ 
\begin{theorem}\label{Lp improving}
	Let $T\subset\R^d$ be a compact set with finite $s$-dimensional upper Minkowski content for $0\leq s<d-1$. Then $A^{T}$ is bounded from $L^p(\mathbb{R}^{d})$ into $L^q(\mathbb{R}^{d})$ for each of the following cases. 
	\begin{enumerate}[label=(\roman*)]
		\item when $d=2$ and $(\frac{1}{p},\frac{1}{q})\in \Delta\setminus\{E,H\}$,
		\item when $d=3$, $s\leq\frac{3}{2}$ and $(\frac{1}{p},\frac{1}{q})\in \Delta\setminus \{E,H\}$,  and for $\frac{3}{2}<s<2$  and  $(\frac{1}{p},\frac{1}{q})\in \Delta\setminus [E,H]$,
		\item when $d\geq4$, $s\leq d-2$ and   $(\frac{1}{p},\frac{1}{q})\in \Delta\setminus \{E,H\}$, and for $d-2<s<d-1$  and  $(\frac{1}{p},\frac{1}{q})\in \Delta\setminus [E,H]$. 
	\end{enumerate}
	Moreover, $A^{T}$ is of restricted weak-type $({p},{q})$, i.e., it is bounded 
	from the Lorentz space $L^{p,1}(\R^d)$ into $L^{q,\infty}(\R^d),$ for each of 
	the following end-points.  
	\begin{enumerate}[label=(\roman*)]
		\item when $d=2$ and $(\frac{1}{p},\frac{1}{q})=E,H$,
		\item when $d=3$, $s\leq\frac{3}{2}$ and $(\frac{1}{p},\frac{1}{q})=E,H$, and $\frac{3}{2}<s<2$ and  $(\frac{1}{p},\frac{1}{q})=E$,
		\item when $d\geq4$, $s\leq d-2$ and  $(\frac{1}{p},\frac{1}{q})=E,H$, and $d-2<s<d-1$ and  $(\frac{1}{p},\frac{1}{q})=E$.
	\end{enumerate}
\end{theorem}
The following remarks point out the sharpness of some indices in~Theorem \ref{Lp improving}.
\begin{rem}
	For $0<s\leq 1$, the restricted weak-type estimate at the point $H$ is sharp in the sense that  $A^T$ does not map $L^{p,r}(\R^d)$ into $L^{q,\infty}(\R^d)$ boundedly for any $r>1$ and $(\frac{1}{p},\frac{1}{q})=H$. We will provide an example in Section \ref{Sec:nec} to support this assertion. 
\end{rem}

\begin{rem}\label{sharpnessatE}
	The point $E$ in Theorem~\ref{Lp improving} is sharp. This can be verified by taking  $T=\{0\}^{d-\lceil s\rceil}\times C_s,$ where $C_s\subset\R^{\lceil s\rceil}$ is a self similar $s$-dimensional set. By taking $f=\chi_{B(0,\delta)}$ for a small $\delta>0,$ we can show that the operator $A^{T}$ is unbounded from $L^p(\R^d)$ to $L^q(\R^d)$ if 
	\[\frac{d}{(d-1)p}-\frac{1-s}{(d-1)q}>1.\]
	This shows the sharpness of the point $E$. 
\end{rem}
\subsection{Results for Nikodym maximal functions associated with sphere}\label{Sec:Nikodym}
The study of Kakeya and Nikodym maximal functions is a classical topic in harmonic analysis and 
geometric measure theory. These objects have been greatly studied to understand some of the most 
interesting phenomena in harmonic analysis, which include the Fourier restriction conjecture and 
the Bochner-Riesz problem. We refer the reader 
to~\cite{TheKakeyaMaximalFunctionAndTheSphericalSummationMultipliers, 
BesicovitchTypeMaximalOperatorsAndApplicationsToFourierAnalysis, 
MattilaGeometryOfSetsMeasureEucli} for more details on these maximal functions. 
Recently, in~\cite{NikodymSetsAndMaximalFunctionsAssociatedWithSpheres}, the authors 
introduced the Nikodym maximal function associated with spheres and studied its $L^p$-boundedness 
properties. In this article, we are concerned with the lacunary analogues of these maximal 
functions in linear and bilinear setting. 

Let $0<\delta<\frac{1}{2}$, the Nikodym maximal function $N^{\delta}$ associated with sphere 
$\mathbb{S}^{d-1}$ is defined by 
$$N^{\delta}f(x):=\sup_{u\in \mathbb{S}^{d-1}}\frac{1}{|{S}^{\delta}(0)|}
\Big|\int_{{S}^{\delta}(0)}f(x+u+y)~dy\Big|,$$ where ${S}^{\delta}(0)$ denotes 
the $\delta$- neighborhood  of the unit sphere $\mathbb{S}^{d-1}$. 
The $L^{p}$-estimates for the operator $N^{\delta}$ are studied in 
\cite{NikodymSetsAndMaximalFunctionsAssociatedWithSpheres}.  

Consider the corresponding lacunary maximal function defined by 
$$N^{\delta}_{lac}f(x):=\sup_{u\in \mathbb{S}^{d-1}}\sup_{k\in\mathbb{Z}}\frac{1}{|{S}^{\delta}(0)|}\Big|\int_{{S}^{\delta}(0)}f(x+2^{k}(u+y))~dy\Big|.$$  
We have the following bounds for the operator $N^{\delta}_{lac}$.
\begin{theorem}\label{Nikodymlacunary}
	Let $0<\delta<\frac{1}{2}$ and $\epsilon>0$. Then the operator $N^{\delta}_{lac}$ satisfies the following estimates
	\begin{enumerate}[label=(\roman*)]
		\item For $d=2$ we have  
		\begin{equation*}
			\|N^\delta_{lac}\|_p\lesssim \begin{cases}
				\delta^{1-\frac{2}{p}-\epsilon}\|f\|_p & \text{if } 1<p\leq2,\\
				\delta^{-\epsilon}\|f\|_p & \text{if } 2\leq p<\infty.
			\end{cases}
		\end{equation*}
		\item For $d=3$ we have  
		\begin{equation*}
			\|N^\delta_{lac}\|_p\lesssim \begin{cases}
				\delta^{\frac{3}{2}-\frac{5}{2p}-\epsilon}\|f\|_p & \text{if } 1<p\leq\frac{3}{2},\\
				\delta^{\frac{1}{2}-\frac{1}{p}-\epsilon}\|f\|_p & \text{if } \frac{3}{2}<p\leq 2,\\
				\delta^{-\epsilon}\|f\|_p & \text{if } 2\leq p<\infty.
			\end{cases}
		\end{equation*}
		\item Finally, for $ d\geq 4$ we have 
		\begin{equation*}
			\|N^\delta_{lac}\|_p\lesssim \begin{cases}
				\delta^{1-\frac{2}{p}-\epsilon}\|f\|_p & \text{if } 1<p\leq\frac{4}{3},\\
				\delta^{1-\frac{2}{p}-\epsilon}\|f\|_p & \text{if } \frac{4}{3}<p\leq2,\\
				\delta^{-\epsilon}\|f\|_p  & \text{if } 2\leq p<\infty.
			\end{cases}
		\end{equation*}
	\end{enumerate}
\end{theorem}
\section{Main results for bilinear maximal functions}\label{Sec:bilinear}
In this section we introduce bilinear analogues of Nikodym maximal function and 
uncentered spherical maximal function.  
\subsection{Bilinear Nikodym maximal function $\mathcal N^\delta$:}\label{Sec:bilinearNikodym}  
We consider the bilinear analogue of the Nikodym maximal function $N^\delta$ defined by 
\[\mathcal N^\delta (f_1,f_2)(x):=\sup_{u,v\in\mathbb S^{d-1}}\frac{1}{|S^\delta(0)|}\left|\int_{S^\delta(0)} 
f_1(x+u+y)f_2(x+v+z)\;d(y,z)\right|,\]
where $0<\delta<\frac{1}{2}$ and $S^\delta(0)=\{(y,z)\in\R^{2d}:\;1-\delta<|(y,z)|<1+\delta\}.$

The following $L^p$-estimates for the operator $\mathcal N^\delta$ hold. 
\begin{thm}\label{Ndelta}
	Let $1\leq p_1,p_2,p_3\leq \infty$ be such that $\frac{1}{p_1}+\frac{1}{p_2}=\frac{1}{p_3}.$ Then 
	\begin{enumerate}[label=(\roman*)]
		\item	$\|\mathcal N^\delta\|_{L^{p_1}\times L^{p_2}\to L^{p_3}}\lesssim 1$ if $(p_1,p_2)\neq (1,1).$ 
		\item	$\|\mathcal N^\delta\|_{L^{1}\times L^{1}\to L^{\frac{1}{2}}}\lesssim \delta^{-1}.$
	\end{enumerate}
\end{thm}

\subsection{Bilinear uncentered spherical maximal functions}
Let  $T\subset \R^d, d\geq 1,$ be a compact set with finite $s$-dimensional upper Minkowski content, where 
$0\leq s\leq d$. Observe that here we allow $d=1$ as well. Define the   uncentered bilinear spherical 
average of scale $t$ associated with $T$ by
\[\mathcal N_{t}^T (f_1,f_2)(x):=\sup_{u,v\in T}\mathcal{A}^{u,v}_t(f_1,f_2)(x):=\sup_{u,v\in T}\left|\int_{\mathbb{S}^{2d-1}} f_1(x+t(u+y))f_2(x+t(v+z))\;d\sigma_{2d-1}(y,z)\right|.\] 
Note that this can be rewritten in terms of Fourier transform as follows 
$$\mathcal N_t^T (f_1,f_2)(x)=\sup_{u,v\in T}\Big|\int_{\mathbb{R}^{2d}}
\widehat{\sigma}_{2d-1}(t\xi,t\eta)e^{2\pi\iota t(u\cdot\xi+v\cdot\eta)}
\widehat{f_1}(\xi)\widehat{f_2}(\eta)e^{2\pi\iota x\cdot(\xi+\eta)}~d\xi d\eta\Big|,$$
where $\widehat{\sigma}_{2d-1}$ denotes the Fourier transform of surface measure 
of the sphere $\mathbb{S}^{2d-1}$. 

We define the corresponding full and lacunary uncentered bilinear spherical 
maximal functions as
\[\mathcal N_{full}^T (f_1,f_2)(x):=\sup_{t>0}\;\mathcal N_{t}^T(f_1,f_2)(x), \hspace{0.7cm}\text{and}\hspace{0.7cm}\mathcal N_{lac}^T (f_1,f_2)(x):=\sup_{k\in\mathbb{Z}}\;\mathcal N_{2^k}^T(f_1,f_2)(x).\]
Observe that if $T=\{\vec{0}\},$ the sharp $L^p$-bounds for H\"older exponents 
are obtained for the maximal function $\mathcal N_{full}^{\vec 0}$ in 
\cite{MaximalEstimatesForTheBilinearSphericalAveragesAndTheBilinearBochnerRieszOperators, BCSS} by the method of slicing for dimensions $d\geq2$. We also refer to \cite{LeeShuin} for bilinear maximal functions defined on degenerate surfaces.

\sloppy For $T=\{\vec 0\}$, Borges and 
Foster~\cite{Borges} and Cho, Lee and Shuin~\cite{CLS} have established $L^{p_1}(\R^d)\times L^{p_2}(\R^d)\rightarrow L^{p_3}(\R^d)$-boundedness of $\mathcal{N}^{\vec 0}_{lac}$ for $d\geq2,$ except for border line cases.

We extend these results to the context of the lacunary maximal function 
$\mathcal N_{lac}^T$. To state the bounds we need the following notation. Given a set of point $\{X_1,X_2,\dots,X_k\}$ in $\R^2$, let $\Omega(\{X_1,X_2,\dots,X_k\})$ denote the open convex hull of all the points $X_1,X_2,\dots,X_k$. We define the points $O=(0,0)$, $A=(1,0)$, $B=(0,1)$, $P=P(d,s),\;Q=Q(d,s)$ and $R=R(d,s)$ as follows,
\[P(d,s)=\left(\frac{d-1}{d-1+s},\frac{d-1}{d-1+s}\right),\;d\geq2,\;0<s<d-1,\]
and for $d\geq3,\;0<s<d-2$,
\[Q(d,s)=\left(1,\frac{d-s-1}{d-s}\right), 
\text{ and }R(d,s)=\left(\frac{d-s-1}{d-s},1\right).\]
We have the following $L^p$-estimates for the operator $\mathcal N_{lac}^T$ in dimension $d\geq 2$. 
\begin{theorem}\label{bilinearNT}
	Let $d\geq2$, $0<s< d-1$ and $0<p_1,p_2,p_3\leq\infty$ with $\frac{1}{p_1}+\frac{1}{p_2}=\frac{1}{p_3}$. 
	Then the operator $\mathcal N_{lac}^T$ is bounded from $L^{p_1}(\mathbb{R}^{d})\times L^{p_2}(\mathbb{R}^{d})$ to $L^{p_3}(\mathbb{R}^{d})$ when 
	\begin{enumerate}
		\item $d\geq2$, $0<s< d-1$, and $(\frac{1}{p_1},\frac{1}{p_2})\in \Omega(\{O,A,P,B\})$.
		\item $d\geq3$, $0<s< d-2$, and $(\frac{1}{p_1},\frac{1}{p_2})\in \Omega(\{O,A,Q,P,R,B\})$.
	\end{enumerate}
\end{theorem}
\begin{figure}[H]
	\centering
	\begin{tikzpicture}[scale=3.8]
		\fill[gray] (0,0)--(1,0)--(0.75,0.75)--(0,1)--cycle;
		\fill[lightgray] (0,1)--(0.4,1)--(0.75,0.75)--(1,0.4)--(1,0)--(0.75,0.75)--(0,1);
		\draw[densely dotted]  (0,1) -- (1,0);
		\draw[thin][->]  (0,0)node[left]{$O$} --(1.15,0) node[right]{$1/p_1$};
		\draw[thin][->]  (0,0) --(0,1.2) node[left]{$1/p_2$};
		\draw[densely dotted] (0,1)--(1,1);
		\draw [densely dotted] (1,0) node[below]{$A$} --(1,1);
		\draw[densely dotted] (0,1)node[left]{$B$}--(0.75,0.75)node[right]{$P$}--(1,0);
		\draw[densely dotted] (0.4,1)node[above]{$R$}--(0.75,0.75)--(1,0.4)node[right]{$Q$};
		
	\end{tikzpicture}
	\caption{Region of boundedness of $\mathcal{N}_{lac}^T$ in Theorem \ref{bilinearNT}.}
	\label{Bilinear lacunary}
\end{figure}
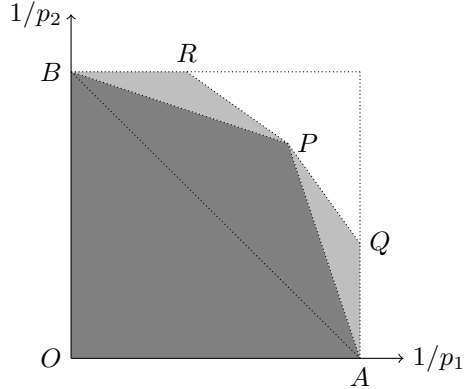
\begin{remark} Observe that in Theorem~\ref{bilinearNT}, the boundedness of $\mathcal{N}_{lac}^T$ for the 
one-dimensional case is not addressed. We would like to remark here that for the case $d=1$, the $L^p$-estimates for 
the operator $\mathcal{N}_{lac}^T$ 
hold in the Banach range. Moreover, these results extend to non-Banach range as well. 
However, the range in the non-Banach case is not quantifiable. Therefore, we present the one-dimensional 
results as a discussion in Section~\eqref{onedimensional}. 
\end{remark}
\section{Proof of Theorem \ref{bilinearNT}: Lacunary bilinear maximal operator $\mathcal N_{lac}^T$ in dimension $d\geq 2$}\label{Sec:prooflacbilinearuncentered}
We will employ a multiscale decomposition of the operator under consideration. 
Let $\phi\in \mathcal S(\R^d)$ be a function such that $\widehat\phi$ is supported 
in $B(0,2)$ and $\widehat\phi(\xi)=1$ for $\xi\in B(0,1)$. 
We define $\widehat{\phi_t}(\xi)=\widehat\phi(t\xi)$ and 
$\widehat{\psi}_t(\xi)= \widehat{\phi}(t\xi)-\widehat{\phi}(2t\xi)$. We choose 
the function $\phi$ such that the following identity holds	
\begin{eqnarray}\label{identity}
	\widehat \phi(\xi)+\sum_{j=1}^\infty\widehat \psi_{2^{-j}}(\xi)=1,\;\xi\neq 0.	
\end{eqnarray}
Also, we will need a counting argument that bounds the $L^p$-norm of the
supremum over T by that of certain linear operators with appropriate growth.
\begin{lemma}[Lemma 4.5, \cite{NikodymSetsAndMaximalFunctionsAssociatedWithSpheres}]\label{coveringlemma}
	Let $T$ be a compact subset of $\mathbb{R}^{d}$ and $\{T_j\}$ be the collection of centers of balls of 
	radius $2^{-j}$ covering $T$. Then, for $p\geq1,$ we have  
	$$\Vert \sup_{u\in T}|f\ast\psi_{2^{-j}}\ast\sigma(\cdot+u)|\Vert^{p}_{L^p(\R^d)}\lesssim \sharp T_j \Vert f\ast\psi_{2^{-j}}\ast\sigma\Vert^p_{L^{p}(\R^d)}.$$ 
\end{lemma}
\textbf{Proof of boundedness in the region $\Omega(\{O,A,P,B\})$:} By interpolation, it suffices to prove weak-type estimates at the points $A,B$ and strong type bounds on the line segment $OP$ (excluding $P$). The weak-type bounds at the points $A$ and $B$ follow by the following pointwise inequality
\[\mathcal N^T_{lac}(f_1,f_2)(x)\lesssim\min\{\|f_1\|_\infty M_{HL}f_2(x),\|f_2\|_\infty M_{HL}f_1(x)\}.\]
Indeed, for $d\geq2,$ the estimate above follows using a slicing argument. We have,
\begin{eqnarray*}
	\mathcal{N}^{T}_{lac}(f_1,f_2)(x)
	&=&2\sup_{k\in\mathbb{Z}}\sup_{(u,v)\in T}\Big|\int_{B^{d}(0,1)}f_1(x+2^{k}(u+y))(1-|y|^2)^{\frac{d-2}{2}}\times\\
	&& \hspace{3cm}\int_{\mathbb{S}^{d-1}}f_2(x+2^{k}(v+\sqrt{1-|y|^2}z))d\sigma(z)dy\Big|\\
	&&\lesssim	M_{HL}f_1(x)\Vert f_2\Vert_{{\infty}}.
\end{eqnarray*}
The other estimate follows similarly.

Now, the proof for $\mathcal{N}^{T}_{lac}$ at the point $P=(\frac{d-1}{d-1+s},\frac{d-1}{d-1+s})$ can be 
completed by exploiting Christ and Zhou's argument from \cite{ChristZhou}.
Using the identity \eqref{identity}, we define new operators $P_k, R_k$ such that 
\[\widehat{P_kf}(\xi)=\hat{f}(\xi)\hat{\phi}(2^{-k}\xi), ~~\widehat{R_{k}f}=\hat{f}(\xi)\hat{\psi}(2^{-k}\xi),.\]
Then we have the following decomposition of the bilinear lacunary maximal function $\mathcal{N}^{T}_{lac}(f_1,f_2)$ 
\begin{eqnarray*}
	\mathcal{N}^{T}_{lac}(f_1,f_2)(x)&\leq& \sup_{k\in\mathbb{Z}}\sup_{u,v\in T}\Big|\mathcal{A}^{u,v}_{2^{-k}}(P_kf_1,P_kf_2)(x)\Big|+\sup_{k\in\mathbb{Z}}\sup_{u,v\in T}\Big|\mathcal{A}^{u,v}_{2^{-k}}(P_kf_1,(I-P_k)f_2)(x)\Big|\\
	&+&\sup_{k\in\mathbb{Z}}\sup_{u,v\in T}\Big|\mathcal{A}^{u,v}_{2^{-k}}((I-P_k)f_1,P_kf_2)(x)\Big|+\sum^{\infty}_{n_1,n_2=1}\sup_{k\in\mathbb{Z}}\sup_{u,v\in T}\Big|\mathcal{A}^{u,v}_{2^{-k}}(R_{n_1+k}f_1,R_{n_2+k}f_2)(x) \Big|\\
	&\lesssim& \sup_{k\in\mathbb{Z}}\sup_{u,v\in T}\Big|\mathcal{A}^{u,v}_{2^{-k}}(P_kf_1,P_kf_2)(x)\Big|+\sup_{k\in\mathbb{Z}}\sup_{u,v\in T}\Big|\mathcal{A}^{u,v}_{2^{-k}}(P_kf_1,f_2)(x)\Big|
	\\
	&+&\sup_{k\in\mathbb{Z}}\sup_{u,v\in T}\Big|\mathcal{A}^{u,v}_{2^{-k}}(f_1,P_kf_2)(x)\Big|+\sum^{\infty}_{n_1,n_2=1}\sup_{k\in\mathbb{Z}}\sup_{u,v\in T}\Big|\mathcal{A}^{u,v}_{2^{-k}}(R_{n_1+k}f_1,R_{n_2+k}f_2)(x) \Big|.
\end{eqnarray*}
For $\mathbf{n}=(n_1,n_2)\in \mathbb{N}^{2}$, we denote 
\begin{eqnarray*}
	&&\mathcal{M}^{T}_{\mathbf{n}}(f_1,f_2)(x):=\sup_{k\in\mathbb{Z}}\sup_{u,v\in T}\Big|\mathcal{A}^{u,v}_{2^{-k}}(R_{n_1+k}f_1,R_{n_2+k}f_2)(x)\Big|,\\
	&&\mathcal{S}^{T}_{\mathbf{n}}(f_1,f_2)(x):=\sum_{k\in\mathbb{Z}}\sup_{u,v\in T}\Big|\mathcal{A}^{u,v}_{2^{-k}}(R_{n_1+k}f_1,R_{n_2+k}f_2)(x)\Big|.
\end{eqnarray*}
Clearly, $\mathcal{M}^{T}_{\mathbf{n}}(f_1,f_2)\leq \mathcal{S}^{T}_{\mathbf{n}}(f_1,f_2)$. The following lemmas are crucial to prove desired $L^{p_1}\times L^{p_2}\to L^{p_3}$ boundedness of $\mathcal{N}^{T}_{lac}$.
\begin{lemma}\label{trivialestimates} Let $d\geq2$. Then 
	we have the following pointwise estimates
	\begin{align*}
		\sup_{k\in\mathbb{Z}}\sup_{u,v\in T}\Big|\mathcal{A}^{u,v}_{2^{-k}}(P_kf_1,f_2)(x)\Big|&\lesssim M_{HL}f_1(x)M_{HL}f_2(x),\\
		\sup_{k\in\mathbb{Z}}\sup_{u,v\in T}\Big|\mathcal{A}^{u,v}_{2^{-k}}(f_1,P_kf_2)(x)\Big|&\lesssim M_{HL}f_1(x)M_{HL}f_2(x),\\
		\sup_{k\in\mathbb{Z}}\sup_{u,v\in T}\Big|\mathcal{A}^{u,v}_{2^{-k}}(P_kf_1,P_kf_2)(x)\Big|&\lesssim M_{HL}f_{1}(x)M_{HL}f_2(x).
	\end{align*}
\end{lemma}
\begin{proof}
	We have already shown that 
	\[\sup_{k\in\mathbb{Z}}\sup_{|y|\leq1}\sup_{u\in T}|P_{k}f_1(x-2^{-k}(u+y))|\lesssim M_{HL}f_1(x).\]
	Therefore, using a slicing argument, we have
	\begin{align*}
		\sup_{k\in\mathbb{Z}}\sup_{u,v\in T}\Big|\mathcal{A}^{u,v}_{2^{-k}}(P_kf_1,f_2)(x)\Big|&\leq \sup_{k\in\mathbb{Z}}\sup_{|y|\leq1}\sup_{u\in T}|P_{k}f_1(x-2^{-k}(u+y))|\sup_{k\in\mathbb{Z}}\sup_{v\in T}\int_{\mathbb{S}^{2d-1}}|f_2(x-2^{-k}(v+z))|d\sigma(y,z)\\
		&\lesssim M_{HL}f_1(x)\sup_{k\in\mathbb{Z}}\sup_{v\in T}\int_{B^{d}(0,1)}|f_2(x-2^{-k}(v+z))|(1-|z|^{2})^{(d-2)/2}\int_{\mathbb{S}^{d-1}}d\sigma(y)~dz\\
		&\leq M_{HL}f_1(x)\sup_{k\in\mathbb{Z}}\int_{B^{d}(0,a)}|f_2(x-2^{-k}z)|~dz\\
		&\lesssim M_{HL}f_1(x)M_{HL}f_2(x).
	\end{align*}
	Here, we have used the fact that the set $T$ is compact. Similarly, we can obtain the other two pointwise estimates.
\end{proof}

\begin{lemma}\label{STn}
	Let $0<s<d-1$. Then for any $\mathbf{n}\in\mathbb{N}^2$ we have 
	\[\Vert \mathcal{M}^{T}_{\mathbf{n}}(f_1,f_2)\Vert_{L^{1}}\lesssim 2^{-(n_1+n_2)(d-1-s)/2}\Vert f_1\Vert_{L^{2}}\Vert f_2\Vert_{L^{2}}.\]
\end{lemma}
\begin{lemma}\label{MTn}
	Let $0<s\leq d$. Then for any $\mathbf{n}\in\mathbb{N}^2$ we have
	\[\Vert \mathcal{M}^{T}_{\mathbf{n}}(f_1,f_2)\Vert_{L^{\frac{1}{2},\infty}}\lesssim 2^{(n_1+n_2)s}(1+|\mathbf{n}|^2)\Vert f_1\Vert_{L^{1}}\Vert f_2\Vert_{L^{1}}.\]
\end{lemma}
Now, for the moment if we assume that the above two lemmas hold true, then using real interpolation and summing in $\mathbf{n}$, we get the desired boundedness of $\mathcal{N}^{T}_{lac}$.

Before giving the proof of Lemmas \ref{STn} and \ref{MTn}, we will prove similar results for a fixed scale version of maximal averages $\mathcal{A}^{u,v}_1$.
\begin{lemma}\label{mainests}
	Let $d\geq2$ and $0<s<d-1$. Then, for $\mathbf{n}=(n_1,n_2)\in\N^2$, we have 
	\begin{enumerate}[label=(\roman*)]
		\item $\left\|\mathcal{A}^{u,v}_{1}(R_{n_1}f_1,R_{n_2}f_2)\right\|_1 \lesssim 2^{-(n_1+n_2)(\frac{d-1-s}{2})}\|R_{n_1}f_1\|_2\|R_{n_2}f_2\|_2,$\label{1inlemma}
		\item $\left\|\mathcal{A}^{u,v}_{1}(R_{n_1}f_1,R_{n_2}f_2)\right\|_{\frac{1}{2}}\lesssim 2^{(n_1+n_2)s}\|f_1\|_1\|f_2\|_1.$ \label{2inlemma}
	\end{enumerate}
\end{lemma}
\begin{proof}
	First, we use the slicing argument to get that
	\begin{align*}
		\mathcal{A}^{u,v}_{1}(R_{n_1}f_1,R_{n_2}f_2)(x)
		&=\sup_{(u,v)\in T}\bigg|\int_0^1r^{d-1}(1-r^2)^{\frac{d-2}{2}}\int_{\mathbb{S}^{d-1}} f_1*\psi_{2^{-n_1}}(x+u+ry)\;d\sigma(y)\\
		&\hspace{20mm}\int_{\mathbb{S}^{d-1}}f_2*\psi_{2^{-n_2}}(x+v+\sqrt{1-r^2}z)\;d\sigma(z)\;dr\bigg|\\
		&\leq \int_0^1r^{d-1}(1-r^2)^{\frac{d-2}{2}}\sup_{u\in T}\left|\int_{\mathbb{S}^{d-1}} f_1*\psi_{2^{-n_1}}(x+u+ry)\;d\sigma(y)\right|\\
		&\hspace{20mm}\sup_{v\in T}\left|\int_{\mathbb{S}^{d-1}}f_2*\psi_{2^{-n_2}}(x+v+\sqrt{1-r^2}z)\;d\sigma(z)\right|\;dr.\\
	\end{align*}
	We employ Cauchy-Schwartz inequality and Lemma \ref{coveringlemma} to obtain,
	\begin{align*}
		\|\mathcal{A}^{u,v}_{1}(R_{n_1}f_1,R_{n_2}f_2)\|_1&\lesssim \int_0^1r^{d-1}(1-r^2)^{\frac{d-2}{2}} 2^{(n_1+n_2)\frac{s}{2}}\|f_1*\psi_{2^{-n_1}}*\sigma_r\|_2\|f_2*\psi_{2^{-n_2}}*\sigma_{\sqrt{1-r^2}}\|_2 \;dr\\
		&\lesssim 2^{-(n_1+n_2)\frac{(d-1-s)}{2}}\|f_1*\psi_{2^{-n_1}}\|_2\|f_2*\psi_{2^{-n_2}}\|_2\int_0^1(1-r^2)^{-\frac{1}{2}}\;dr\\
		&\lesssim 2^{-(n_1+n_2)\frac{(d-1-s)}{2}}\|f_1*\psi_{2^{-n_1}}\|_2\|f_2*\psi_{2^{-n_2}}\|_2.
	\end{align*}
	
	To prove estimate \ref{2inlemma}, we rely on a covering argument for the compact set $T$ modified 
	for spatially localized Littlewood-Paley pieces of the functions $f_1$ and $f_2$. Let $\Psi_{2^{-j}}(x)=2^{jd}(2+2^j|x|)^{-(d+1)}$ so that $|\psi_{2^{-j}}(x)|\lesssim\Psi_{2^{-j}}(x)$. Also let $\R^d=\cup_{Q\in\mathfrak C}Q$, where $Q\in\mathfrak{C}$ are unit cubes with sides parallel to axes and $aQ$ is the concentric cube containing $Q$ with side length $a$, where the absolute parameter $a>2$ is such that the compact set $T\subset B(0,\frac{a}{2}-1)$. Then, we have
	\begin{align*}
		&\|\mathcal{A}^{u,v}_{1}(R_{n_1}f_1,R_{n_2}f_2)\|_\frac{1}{2}^\frac{1}{2}\\
		&=\int\left(\sup\limits_{u,v}\Big|\sum\limits_{Q\in\mathfrak{C}}((f_1*\psi_{2^{-n_1}})\chi_{_Q}\otimes( f_2*\psi_{2^{-n_2}}))*\sigma(x+u,x+v)\Big|\right)^\frac{1}{2}dx\\
		&\leq\sum_{Q\in\mathfrak{C}}\int_{aQ}\left(\sup\limits_{u,v}((|f_1|*\Psi_{2^{-n_1}})\chi_{_Q}\otimes( |f_2|*\Psi_{2^{-n_2}})\chi_{_{3aQ}})*\sigma(x+u,x+v)\right)^\frac{1}{2}dx\\
		&\lesssim\sum_{Q\in\mathfrak{C}}\left(\int_{aQ}\sup\limits_{u,v}((|f_1|*\Psi_{2^{-n_1}})\chi_{_Q}\otimes( |f_2|*\Psi_{2^{-n_2}})\chi_{_{3aQ}})*\sigma(x+u,x+v)\;dx\right)^\frac{1}{2}\\
		&\leq\sum_{Q\in\mathfrak{C}}\left(\int_{aQ}\int_{\s^{2d-1}}\sup\limits_{u}\;\left((|f_1|*\Psi_{2^{-n_1}})\chi_{_Q}\right)(x+u-y)\;\sup\limits_{v}\;\left(( |f_2|*\Psi_{2^{-n_2}})\chi_{_{3aQ}}\right)(x+v-z)d\sigma(y,z)\;dx\right)^\frac{1}{2}\\
	\end{align*}
	Let $C_1$ be the set of centers of balls of radius $2^{-n_1}$ covering $T$ and $C_2$ be the set of centers of 
	balls of radius $2^{-n_2}$ covering $T$. We note that $\Psi_{2^{-n_1}}(w_1)\lesssim\Psi_{2^{-n_1}}(w_2)$ for 
	$|w_1-w_2|\leq2^{-n_1}$. Then, we have
	\begin{align*}
		&\|\mathcal{A}^{u,v}_{1}(R_{n_1}f_1,R_{n_2}f_2)\|_\frac{1}{2}^\frac{1}{2}\\
		&\lesssim\sum_{Q\in\mathfrak{C}}\left(\int_{aQ}\int_{\s^{2d-1}}\sum\limits_{u\in C_1}\;\left((|f_1|*\Psi_{2^{-n_1}})\chi_{_{2Q}}\right)(x+u-y)\sum\limits_{v\in C_2}\;\left(( |f_2|*\Psi_{2^{-n_2}})\chi_{_{4aQ}}\right)(x+v-z)d\sigma(y,z)\;dx\right)^\frac{1}{2}\\
		&\lesssim\sum_{Q\in\mathfrak{C}}\left(\sum\limits_{u\in C_1,v\in C_2}\;\|\tau_{-u}\left((|f_1|*\Psi_{2^{-n_1}})\chi_{_{2Q}}\right)\|_1\;\|\tau_{-v}\left(( |f_2|*\Psi_{2^{-n_2}})\chi_{_{4aQ}}\right)\|_1\right)^\frac{1}{2}\\
		&\leq2^{(n_1+n_2)\frac{s}{2}}\left(\sum_{Q\in\mathfrak{C}}\|(|f_1|*\Psi_{2^{-n_1}})\chi_{_{2Q}}\|_1\right)^\frac{1}{2}\left(\sum_{Q\in\mathfrak{C}}\|(|f_2|*\Psi_{2^{-n_2}})\chi_{_{4aQ}}\|_1\right)^\frac{1}{2}\\
		&\lesssim2^{(n_1+n_2)\frac{s}{2}}\|f_1\|_1^\frac{1}{2}\|f_2\|_1^\frac{1}{2},
	\end{align*}
	where we have used the $L^1\times L^1\to L^1$-boundedness of single bilinear spherical average obtained 
	in \cite{IPS} in the second step and a Cauchy-Schwartz inequality in the third step.
	It can also be proved that \[\|\mathcal{A}^{u,v}_{1}(R_{n_1}f_1,R_{n_2}f_2)\|_{\frac{1}{2}}\lesssim2^{(i+j)s}\|f_1\ast\tilde{\psi}_{2^{-n_1}}\|_1\|f_2\ast\tilde{\psi}_{2^{-n_2}}\|_1.\]
	The above inequality follows from the fact that $\tilde{\psi}_{2^{-n_2}}\ast\psi_{2^{-n_2}}=\psi_{2^{-n_2}}$, 
	where $supp(\widehat{\psi_{2^{-n_2}}})\subset supp(\widehat{\tilde{\psi}_{2^{-n_2}}})$ and 
	$\widehat{\tilde{\psi}_{2^{-n_2}}}(\xi)=1$ for $\xi\in supp(\widehat{\psi_{2^{-n_2}}})$.
\end{proof}

\begin{proof}[Proof of Lemma \ref{STn}]
	Note that from Lemma \ref{mainests} we have 
	\[\Vert\sup_{u,v\in T}\Big|\mathcal{A}^{u,v}_1(R_{n_1}f_1,R_{n_2}f_2)\Big|\Vert_{L^{1}}\lesssim 2^{-(n_1+n_2)(d-1-s)/2}\Vert f_1\Vert_{L^{2}}\Vert f_2\Vert_{L^{2}}.\]
	We write 
	\begin{eqnarray*}
		\mathcal{A}^{u,v}_{2^{-k}}(R_{n_1+k}f_1,R_{n_2+k}f_2)(x)=\mathcal{A}^{u,v}_1(R_{n_1}\delta_{2^{-k}}f_1,R_{n_2}\delta_{2^{-k}}f_2)(2^{k}x),
	\end{eqnarray*}
	where $\delta_{2^{-k}}f(x)=f(2^{-k}x)$.
	Also we have the following identity 
	$\tilde{R}_{n+k}\circ R_{n+k}=R_{n+k}$, where $\widehat{\tilde{R}_{n+k}f}(\xi)=\hat{f}(\xi)\hat{\tilde{\psi}}_{2^{-(n+k)}}(\xi)$.
	Therefore, using Lemma \ref{mainests} we get 
	\begin{eqnarray*}
		\Vert \mathcal{M}^{T}_{\mathbf{n}}(f_1,f_2)\Vert_{L^{1}}&\leq& \sum_{k\in\mathbb{Z}}\int \sup_{u,v\in T}\Big|\mathcal{A}^{u,v}_{2^{-k}}(R_{n_1+k}\circ \tilde{R}_{n_1+k}f_1,R_{n_2+k}\circ \tilde{R}_{n_2+k}f_2)(x)\Big| dx\\
		&\lesssim& 2^{-(n_1+n_2)(d-1-s)/2}\sum_{k\in\mathbb{Z}} \Vert \tilde{R}_{n_1+k}f_1\Vert_{L^{2}}\Vert \tilde{R}_{n_2+k}f_2\Vert_{L^{2}}\\
		&\leq& 2^{-(n_1+n_2)(d-1-s)/2} \Vert f_1\Vert_{L^{2}} \Vert f_2\Vert_{L^{2}}.
	\end{eqnarray*}
	The last inequality follows from Cauchy-Schwartz inequality and Plancherel's identity.
\end{proof}
\begin{proof}[Proof of Lemma \ref{MTn}]
	We have 
	\begin{eqnarray*}
		\mathcal{M}^{T}_{\mathbf{n}}(f_1,f_2)(x)&=&\sup_{k\in \mathbb{Z}}\Big|\mathcal{A}^{T}_{2^{-k}}(R_{n_1+k}f_1,R_{n_2+k}f_2)(x)\Big|\\
		&=&\sup_{k\in \mathbb{Z}}\Big|\mathcal{A}^{T}_{2^{-k}}(R_{n_1+k}\circ\tilde{R}_{n_1+k}f_1,R_{n_2+k}\circ\tilde{R}_{n_2+k}f_2)(x)\Big|.
	\end{eqnarray*}
	Our claim is the following 
	\[\Big|\{x:\mathcal{M}^{T}_{\mathbf{n}}(f_1,f_2)(x)>\alpha\}\Big|\leq C|\mathbf{n}|^2\alpha^{-1/2}.\]
	We shall use Calder\'{o}n-Zygmund decomposition for both the functions $f,g$. Let $c_0>0$ be a constant which will be chosen later. Also let $\alpha>0$ and $\Vert f_j\Vert_{L^{1}}=1$. We decompose the functions as $f_j=g_j+h_j$, for $j=1,2$ with 
	\begin{eqnarray*}
		&&\Vert g_j\Vert_{L^{\infty}}\leq c_0\alpha, ~~h_j=\sum_{\beta_j}h_{\beta_j}, ~\text{where}~supp(h_{\beta_j})\subset Q_{\beta_j} ~~\text{and they are disjoint dyadic cubes,}~\text{and}~\\
		&&\Vert h_{\beta_j}\Vert_{L^{1}}\leq C\alpha|Q_{\beta_j}| ~~ \text{with}~~ \int_{Q_{\beta_j}}h_{\beta_j}=0. 
	\end{eqnarray*}
	In the above $\beta_1,\beta_2$ are two indexing sets.
	Define 
	\[\mathcal{E}:=\Big(\cup_{\beta_1}\tilde{Q}_{\beta_1}\Big)\cup \Big(\cup_{\beta_2}\tilde{Q}_{\beta_2}\Big),\]
	where $\tilde{Q}$ is a concentric cube with side length $l(\tilde{Q})=10l(Q)$. From the Calder\'{o}n--Zygmund decomposition it follows that $|\mathcal{E}|\lesssim \alpha^{-1/2}$.\\
	\textbf{Contribution from $\mathcal{M}^{T}_{\mathbf{n}}(h_1,h_2)$:}
	It is enough to show that 
	\[\int_{\mathbb{R}^d\setminus \mathcal{E}}|\mathcal{M}^{T}_{\mathbf{n}}(h_1,h_2)(x)|^{1/2}\lesssim |\mathbf{n}|^{2}.\]
	We further decompose each $h_j$ as follows 
	\[h_{j}=\sum_{i\in\mathbb{Z}}h^{i}_{j},~~\text{where}~~h^{i}_{j}=\sum_{l(Q_{\beta_j})=2^{-i}}h_{\beta_j}.\]
	Then we have the following estimate.
	\begin{lemma}\label{badfunction}
		Let $k,i_1,i_2\in\mathbb{Z}$ and $n_1,n_2\in\mathbb{N}$. Then the following estimate holds uniformly in $k,i_j,n_j$
		\begin{eqnarray*}
			\int_{\mathbb{R}^{d}\setminus \mathcal{E}}\Big|\mathcal{A}^{T}_{2^{-k}}(R_{n_1+k}h^{i_1}_1,R_{n_2+k}h^{i_2}_2)\Big|^{1/2}\lesssim 2^{(n_1+n_2)s/2}\min_{j=1,2}\min\{1,2^{i_j-k},2^{(k+n_j-i_j)/2}\}\Vert h^{i_1}_{1}\Vert^{1/2}_{L^{1}}\Vert h^{i_2}_{2}\Vert^{1/2}_{L^{1}}.
		\end{eqnarray*} 
	\end{lemma}
	\begin{proof}
		Note that from $(ii)$ of Lemma \ref{mainests} we have
		\begin{eqnarray*}
			\int_{\mathbb{R}^{d}\setminus \mathcal{E}}\Big|\mathcal{A}^{T}_{2^{-k}}(R_{n_1+k}h^{i_1}_1,R_{n_2+k}h^{i_2}_2)\Big|^{1/2}&\lesssim &2^{(n_1+n_2)s/2}\Vert \tilde{R}_{n_1+k}h^{i_1}_{1}\Vert^{1/2}_{L^{1}}\Vert \tilde{R}_{n_2+k}h^{i_2}_{2}\Vert^{1/2}_{L^{1}}\\
			&\lesssim& 2^{(n_1+n_2)s/2}\Vert h^{i_1}_{1}\Vert^{1/2}_{L^{1}}\Vert h^{i_2}_{2}\Vert^{1/2}_{L^{1}}.
		\end{eqnarray*}
		Moreover, when $i_j>n_j+k$, we use the mean zero property of the functions $h_{\beta_j}$. Note that $h^{i_j}_j=\sum_{l(Q_{\beta_j})=2^{-i_j}}h_{\beta_j}$ . Therefore, 
		\begin{eqnarray*}
			\tilde{R}_{n_j+k}h_{\beta_j}(x)&=&\int_{Q_{\beta_j}} \Big(\tilde{\psi}_{2^{-(n_j+k)}}(x-y)-\tilde{\psi}_{2^{-(n_j+k)}}(x-c_{Q})\Big)h_{\beta_j}(y)dy\\
			&=&\int_{\mathbb{R}^{d}}\int^{1}_{0}\langle \nabla_{y}\tilde{\psi}_{2^{-(n_j+k)}}(x-c_{Q}-t(y-c_{Q})),y-c_{Q}\rangle dt~h_{\beta_j}(y) dy, 
		\end{eqnarray*}
		where $c_Q$ is the center of $Q_{\beta_j}$. Observe that 
		\[\Big||y-c_{Q}|^{-1}\langle \nabla_{y}\tilde{\psi}_{2^{-(n_j+k)}}(x-c_{Q}-t(y-c_{Q})),y-c_{Q}\rangle\Big|\lesssim 2^{(n_j+k)(d+1)}\Big(1+|x-c_{Q}-t(y-c_{Q})|\Big)^{-N},\]
		for any $N>1$. Note that $x\notin \mathcal{E}$, therefore $|x-c_{Q}-t(y-c_{Q})|\gtrsim |x-c_Q|$. Now, applying Minkowski's integral inequality we get 
		\begin{eqnarray*}
			\Vert \tilde{R}_{n_j+k}h^{i_j}_{j}\Vert_{L^{1}}\lesssim 2^{(n_j+k)} 2^{-i_j}\Vert h^{i_j}_{j}\Vert_{L^{1}}.
		\end{eqnarray*}
		On the other hand, when $i_j<k$, we need to run the proof of $(ii)$ of Lemma \ref{mainests} with slight modifications. Note that $x\notin \mathcal{E}$. 
		Let $k\in \mathbb{Z}$ and w.l.o.g assume that $0>k>i_1$. Also let $\R^d=\cup_{Q\in\mathfrak C}Q$, where $Q\in\mathfrak{C}$ are cubes with side length $2^{-k}$ and sides parallel to the coordinate axes, and $aQ$ is the concentric cube containing $Q$ with side length $a2^{-k}$, where the absolute parameter $a>2$ is such that the compact set $T\subset B(0,\frac{a}{2}-1)$. 
		\begin{eqnarray*}
			&&\|\mathcal A^{T}_{2^{-k}}(R_{n_1+k}h^{i_1}_1,R_{n_2+k}h^{i_2}_2)\|_{L^{1/2}(\mathbb{R}^{d}\setminus\mathcal{E})}^\frac{1}{2}\\	
			&&=\int_{\mathbb{R}^{d}\setminus\mathcal{E}}\left(\sup\limits_{u,v\in T}\Big|\sum\limits_{Q\in\mathfrak{C}}((h^{i_1}_1*\psi_{2^{-n_1-k}})\chi_{_Q}\otimes( h^{i_2}_2*\psi_{2^{-n_2-k}}))*\sigma_{2^{-k}}(x+2^{-k}u,x+2^{-k}v)\Big|\right)^\frac{1}{2}dx\\
			&&\leq\sum_{Q\in\mathfrak{C}}\int_{aQ\setminus\mathcal{E}}\left(\sup\limits_{u,v\in T}((|h^{i_1}_1|*|\psi_{2^{-n_1-k}}|)\chi_{_Q}\otimes( |h^{i_2}_2|*\Psi_{2^{-n_2-k}})\chi_{_{3aQ}})*\sigma_{2^{-k}}(x+2^{-k}u,x+2^{-k}v)\right)^\frac{1}{2}dx\\
			&&\lesssim\sum_{Q\in\mathfrak{C}}\int_{aQ\setminus\mathcal{E}}\Big[\int_{\s^{2d-1}}\sup\limits_{u}\;\left((|h^{i_1}_1|*\Psi_{2^{-n_1-k}})\chi_{_Q}\right)(x+2^{-k}u-2^{-k}y)\\
			&&\sup\limits_{v}\;\left(( |h^{i_2}_2|*\Psi_{2^{-n_2-k}})\chi_{_{3aQ}}\right)(x+2^{-k}v-2^{-k}z)d\sigma(y,z)\Big]^\frac{1}{2}~dx
		\end{eqnarray*}
		Since $x\notin \mathcal{E}$  and $h^{i_1}_1$ is sum of $h_{\beta_1}$ with $l(Q_{\beta_1})=2^{-i_1}$. Therefore,  $dist(x+2^{-k}u-2^{-k}y, y')\geq 2^{2-i_1}$, as $y'\in supp(h_{\beta_1})$ and   $2^{-k}<2^{-i_1}$.
		Let $C_1$ be the set of centers of balls of radius $2^{-n_1}$ covering $T$ and $C_2$ be the set of centers of balls of radius $2^{-n_2}$ covering $T$. Observe that  $\Psi_{2^{-n_1-k}}(2^{-k}w_1)\lesssim\Psi_{2^{-n_1-k}}(2^{-k}w_2)$ for $|w_1-w_2|\leq2^{-n_1}$ and  $\Psi_{2^{-n_2-k}}(2^{-k}w_1)\lesssim\Psi_{2^{-n_2-k}}(2^{-k}w_2)$ for $|w_1-w_2|\leq2^{-n_2}$. Here we define $\Phi(y)=\Psi(y)\chi_{|y|\geq 2^{-i_1}}$. 
		Then, we have
		\begin{eqnarray*}
			&&\|\mathcal A^{T}_{2^{-k}}(R_{n_1+k}h^{i_1}_1,R_{n_2+k}h^{i_2}_2)\|_\frac{1}{2}^\frac{1}{2}\\
			&&\lesssim\sum_{Q\in\mathfrak{C}}\int_{aQ\setminus\mathcal{E}}\Big[\int_{\s^{2d-1}}\sum\limits_{u\in C_1}\;\left((|h^{i_1}_1|*\Psi_{2^{-n_1-k}})\chi_{_{2Q}}\right)(x+2^{-k}u-2^{-k}y)\\
			&&\sum\limits_{v\in C_2}\;\left(( |h^{i_2}_2|*\Psi_{2^{-n_2-k}})\chi_{_{4aQ}}\right)(x+2^{-k}v-2^{-k}z)d\sigma(y,z)\Big]^\frac{1}{2}\;dx\\
			&&\lesssim\sum_{Q\in\mathfrak{C}}\Big[\sum\limits_{u\in C_1,v\in C_2}\;\sum_{l(Q_{\beta_1})=2^{-i_1}}\|\tau_{-u}\delta_{2^{-k}}\left((|h^{i_1}_{\beta_1}|*\Phi_{2^{-n_1-k}})\chi_{_{2Q}}\right)(2^{k}\cdot)\|_1\\ 
			&&\|\tau_{-v}\delta_{2^{-k}}\left(( |h^{i_2}_2|*\Psi_{2^{-n_2-k}})\chi_{{4aQ}}\right)(2^{k}\cdot)\|_1\Big]^\frac{1}{2}\\
			&&\leq2^{(n_1+n_2)\frac{s}{2}}\left(\sum_{Q\in\mathfrak{C}}\sum_{l(Q_{\beta_1})=2^{-i_1}}\|(|h^{i_1}_{\beta_1}|*\Phi_{2^{-n_1-k}})\chi_{_{2Q}}\|_1\right)^\frac{1}{2}\left(\sum_{Q\in\mathfrak{C}}\|(|h^{i_2}_2|*\Psi_{2^{-n_2-k}})\chi_{_{4aQ}}\|_1\right)^\frac{1}{2}\\
			&&\lesssim2^{(n_1+n_2)\frac{s}{2}} 2^{-N(n_1+k-i_1)}\|h^{i_1}_1\|_1^\frac{1}{2}\|h^{i_2}_2\|_1^\frac{1}{2},
		\end{eqnarray*}
		In the last inequality we have used Young's inequality. The case when $k>0$ and $k>i_j$ can be handled with the similar argument with slight modifications. This completes the proof of Lemma \ref{badfunction}.
	\end{proof}
	Now, invoking Lemma $5.2$ of \cite{ChristZhou} we get 
	\begin{eqnarray}
		\sum_{k\in\mathbb{Z}}\sum_{i_1,i_2\in\mathbb{Z}}\min_{j=1,2}\min\Big(2^{i_j-k}, 2^{k+n_j-i_j},1\Big)\prod^{2}_{j=1}\Vert h^{i_j}_{j}\Vert^{1/2}_{L^{1}}\lesssim |\mathbf{n}|^{2}\prod^{2}_{j=1}\Vert h_{j}\Vert^{1/2}_{L^{1}}.
	\end{eqnarray}

	\textbf{Contribution from $\mathcal{M}^{T}_{\mathbf{n}}(h_1,g_2)$, $\mathcal{M}^{T}_{\mathbf{n}}(g_1,h_2)$ and  $\mathcal{M}^{T}_{\mathbf{n}}(g_1,g_2)$: }\\
	Firstly, we estimate \[|\{x:\mathcal{M}^{T}_{\mathbf{n}}(h_1,g_2)>\alpha/4\}|.\]
	Define $\mathcal{E}_{1}=\cup_{\beta_1}\tilde{Q}_{\beta_1}$ such that $|\mathcal{E}_1|\lesssim \alpha^{-1/2}$. Then using similar and rather simplified argument as above we get 
	\begin{eqnarray*}
		|\{x\in \mathbb{R}^{d}\setminus \mathcal{E}_1:\mathcal{M}^{T}_{\mathbf{n}}(h_1,g_2)(x)>\alpha/4\}|\lesssim |\mathbf{n}|^2\alpha^{-1/2}.
	\end{eqnarray*}
	Therefore, we get \begin{eqnarray*}
		|\{x\in \mathbb{R}^{d}:\mathcal{M}^{T}_{\mathbf{n}}(h_1,g_2)(x)>\alpha/4\}|\lesssim |\mathbf{n}|^2\alpha^{-1/2}+\alpha^{-1/2}.
	\end{eqnarray*}
	Similarly for $\mathcal{M}^{T}_{\mathbf{n}}(g_1,h_2)$ we get the similar estimate. Note that $\Vert g_j\Vert_{L^{\infty}}\leq c_0 \alpha^{1/2}$ and 
	\[\mathcal{M}^{T}_{\mathbf{n}}(g_1,g_2)\leq C_d \Vert g_1\Vert_{L^{\infty}}\Vert g_2\Vert_{L^{\infty}}\leq C_{d}c^{2}_{0}\alpha\]
	Therefore, we choose $c_0>0$ such that $C_dc^{2}_0<1/4$. This implies 
	\[|\{x:\mathcal{M}^{T}_{\mathbf{n}}(g_1,g_2)(x)>\alpha/4\}|=0.\]
	This completes the proof of Lemma \ref{MTn}.
\end{proof}

\textbf{Proof of boundedness in the region $\Omega(\{O,A,Q,P,R,B\})$:} It remains to prove the boundedness in the triangles $\Omega(\{A,Q,P\})$ and $\Omega(\{P,R,B\})$. We prove the boundedness in the region $\Omega(\{A,Q,P\})$ and the other follows by symmetry. Moreover, by interpolation, it is enough to prove weak type bounds on the line segment $AQ$, excluding the point $Q$.

Let $A_l=\{y\in\R^d:2^{-l}<\sqrt{1-|y|^2}\leq2^{-l+1}\}$. By slicing argument, we can write the integral as
\begin{eqnarray*}
	&&\left|\int_{\mathbb{S}^{2d-1}} f_1(x+2^{k}(u+y))f_2(x+2^{k}(v+z))\;d\sigma_{2d-1}(y,z)\right|\\&=&\left|\int_{B(0,1)}f_1(x+2^{k}(u+y))(1-|y|^2)^{\frac{d-2}{2}}\int_{\mathbb{S}^{d-1}}f_2(x+2^{k}(v+\sqrt{1-|y|^2}z))\;d\sigma(z)dy\right|\\
	&=&\left|\sum_{l=1}^{\infty}\int_{A_l}f_1(x+2^{k}(u+y))(1-|y|^2)^{\frac{d-2}{2}}\int_{S^{d-1}}f_2(x+2^{k}(v+\sqrt{1-|y|^2}z))\;d\sigma(z)dy\right|\\
	&\lesssim& \sum_{l=1}^\infty 2^{-l(d-2)}\left(\sup_{2^{-l}<r\leq2^{-l+1}}\left|\int_{\mathbb{S}^{d-1}}f_2(x+2^{k}(v+rz))\;d\sigma(z)\right|\right)\int_{A_l}|f_1(x+2^{k}(u+y))|~dy.
\end{eqnarray*}

Therefore, we get that
\begin{eqnarray*}
	\mathcal N_{lac}^T (f_1,f_2)(x)\lesssim \sum_{l=1}^\infty 2^{-l(d-2)} N_{T,l}(f_2)(x)M_l(f_1)(x),
\end{eqnarray*}
where
\begin{eqnarray*}
	N_{T,l}(f_2)(x)&=&\sup_{k\in\Z}\sup_{v\in T}\sup_{2^{-l}<r\leq2^{-l+1}}\left|\int_{\mathbb{S}^{d-1}}f_2(x+2^k(v+rz))\;d\sigma(z)\right|\\
	&=&\sup_{k\in\Z}\sup_{v\in T}\sup_{2^{-l}<r\leq2^{-l+1}}|f_2\ast\sigma_{2^kr}(x+2^kv)|
\end{eqnarray*}
and
\[M_l(f_1)(x)=\sup_{k\in\Z}\sup_{u\in T}\int_{A_l}|f_1(x+2^{k}(u+y))|~dy.\]
We can see that $M_l$ maps $L^1(\mathbb R^d)$ to $L^{1,\infty}(\mathbb R^d)$ for all $1\leq l< \infty$ with constant independent of $l$. Now, we provide $L^p$-estimates for the intermediary operators $N_{T,l}$.
\begin{lemma}\label{contlemma}
	Let $d\geq 3$ and $\epsilon>0$. Suppose that $T\subset \R^d$ is a compact set with finite upper Minkowski content $s<d-2$. Then 
	\[\|N_{T,l}(f)\|_p\lesssim \max\{2^{l(\frac{2d-(1+\epsilon)s}{p}-d+(1+\epsilon)s)},l2^{ls}\}\|f\|_p \quad \text{for}\quad p> 1+ \frac{1}{d-s-1}.\]
\end{lemma}

\begin{proof}
	Fix $r\in(2^{-l},2^{-l+1}]$ and using identity \ref{identity}, we write $\sigma_{2^kr}$ as
	\[\sigma_{2^kr}=\phi_{2^kr}\ast\sigma_{2^kr}+\sum_{j=1}^{\infty}\psi_{2^{k-j}r}\ast \sigma_{2^kr}.\]
	The above equation gives the following decomposition of $N_{T,l}$.
	\begin{align*}
		N_{T,l}f(x)&\leq \sup_{k\in\Z}\sup_{v\in T}\sup_{2^{-l}<r\leq2^{-l+1}}|f\ast\phi_{2^kr}\ast\sigma_{2^kr}(x+2^kv)|+\sum_{j=1}^\infty\sup_{k\in\Z}\sup_{v\in T}\sup_{2^{-l}<r\leq2^{-l+1}}|f\ast\psi_{2^{k-j}r}\ast\sigma_{2^kr}(x+2^kv)|\\
		&=\sum_{j=0}^\infty N_{T,l}^jf(x),
	\end{align*}
	where
	\begin{align*}
		N_{T,l}^0f(x)&=\sup\limits_{k\in\Z}\sup\limits_{v\in T}\sup\limits_{2^{-l}<r\leq2^{-l+1}}|f\ast\phi_{2^kr}\ast\sigma_{2^kr}(x+2^kv)|\;\text{ and }\\
		N_{T,l}^jf(x)&=\sup\limits_{k\in\Z}\sup\limits_{v\in T}\sup\limits_{2^{-l}<r\leq2^{-l+1}}|f\ast\psi_{2^{k-j}r}\ast\sigma_{2^kr}(x+2^kv)|,\;j\geq1
	\end{align*}
	It is enough to prove $L^p$-estimates for each $N_{T,l}^j$ with some decay in $j$ and small enough uniform growth in $l$ so that the sum in $l$ is finite.
	
	For $j=0$, let $\tilde{\phi}\in \mathcal S(\R^d)$ be such that $\widehat{\tilde{\phi}}(\xi)=1$ on the support of $\hat{\phi}$. Now, Let $\Phi_{2^kr}(x)=\frac{1}{2^{kd}r^d}\left(2+\frac{|x|}{2^kr}\right)^{-(d+1)}$ so that $|\tilde{\phi}_{2^kr}(x)|\lesssim\Phi_{2^kr}(x)$ and note that $\Phi_{2^kr}(w_1)\lesssim\Phi_{2^kr}(w_2)$ whenever $|w_1-w_2|\leq2^kr$. Assume $C$ to be a collection of the center of balls of radius $r$ covering the set $T$. Using the covering argument, we obtain
	\begin{align}
		\sup\limits_{v\in T}|f\ast\phi_{2^kr}\ast\sigma_{2^kr}(x+2^kv)|&=\sup\limits_{v\in T}|f\ast\phi_{2^kr}\ast\tilde{\phi}_{2^kr}\ast\sigma_{2^kr}(x+2^kv)|\nonumber\\
		&\lesssim\sup\limits_{v\in C}|f\ast\phi_{2^kr}\ast\sigma_{2^kr}\ast\Phi_{2^kr}(x+2^kv)|\nonumber\\
		&\lesssim\sum_{v\in C}|f\ast\phi_{2^kr}\ast\sigma_{2^kr}\ast\Phi_{2^kr}(x+2^kv)|.\label{cover}
	\end{align}
	Let $h(x)=\sup\limits_{k\in\Z}\sup\limits_{2^{-l}<r\leq2^{-l+1}}|f\ast\phi_{2^kr}\ast\sigma_{2^kr}(x)|$. Then, we have
	\begin{align*}
		h\ast\Phi_{2^kr}(x+2^kv)&=\int_{\R^d}h(x+2^kv-y)\frac{1}{2^{kd}r^d}\left(2+\frac{|y|}{2^kr}\right)^{-(d+1)}\;dy\\
		&=\int_{\R^d}h(x+2^kv-2^{k-l}y)\frac{1}{2^{ld}r^d}\left(2+\frac{|y|}{2^lr}\right)^{-(d+1)}\;dy\\
		&\lesssim\int_{\R^d}h(x+2^{k-l}(y-2^lv))\left(1+|y|\right)^{-(d+1)}\;dy,\label{shift}
	\end{align*}
	where we have used a change of variable $y\to 2^{k-l}y$ in the second step along with $2^{-l}<r\leq2^{-l+1}$. Thus
	\begin{equation}\label{shift}
		\sup_{k\in\Z}\sup_{2^{-l}<r\leq2^{-l+1}}h\ast\Phi_{2^kr}(x+2^kv)\lesssim\mathfrak{M}_{2^lv}h(x),
	\end{equation}
	where $\mathfrak{M}_{2^lv}$ is the shifted Hardy-Littlewood maximal function with a shift by $2^lv$. 
	Therefore, covering argument \eqref{cover} and equation \eqref{shift} give us
	\begin{align*}
		N_{T,l}^0f(x)&\lesssim\sum_{v\in C}\mathfrak{M}_{2^lv}h(x)\\
		&\lesssim\sum_{v\in C}\mathfrak{M}_{2^lv}(M_{HL}f)(x),
	\end{align*}
	where the final inequality in the above follows from the kernel estimates of $\phi_{2^kr}\ast\sigma_{2^kr}$. Using the $L^p$-estimates of Hardy-Littlewood maximal function and the logarithmic bound for shifted maximal function from \cite{shiftedmaximal}, we obtain
	\[\|N_{T,l}^0f\|_p\lesssim\sum_{v\in C}\log(1+2^l|v|)\|f\|_p,\quad p>1.\]
	Since $T$ is a compact set and $\# C\lesssim2^{ls}$, we have
	\begin{equation}\label{j=0}
		\|N_{T,l}^0f\|_p\lesssim l2^{ls}\|f\|_p,\quad p>1.
	\end{equation}
	
	For $j\geq1$, we have the following kernel estimate
	\[|\psi_{2^{k-j}r}\ast \sigma_{2^kr}(x+2^kv)|\lesssim\frac{2^j}{2^{kd}r^d}\left(1+\frac{|x+2^kv|}{2^kr}\right)^{-N}.\]
	The above kernel estimate gives us
	\begin{equation}\label{NjL1}
		N_{T,l}^j f(x)\lesssim 2^j2^{ld}M_{HL}f(x).
	\end{equation}
	
	Next, we prove the $L^2$-estimates using a covering argument and Littlewood-Paley theory. Let $C'$ be the center of the balls of radius $2^{-(j+l)}$ covering the set $T$ and recall that $\Psi_{2^{-j}r}(w_1)\lesssim\Psi_{2^{-j}r}(w_2)$ whenever $|w_1-w_2|\leq2^{-j}r$. Then, for $\epsilon>0$
	\begin{align*}
		\sup_{v\in T}|f\ast\psi_{2^{-j}r}\ast\sigma_r(x+v)|^{\frac{2}{1+\epsilon}}&\lesssim\sup_{v\in C'}|f\ast\psi_{2^{-j}r}\ast\sigma_r|^{\frac{2}{1+\epsilon}}\ast\Psi_{2^{-j}r}(x+v)\\
		&\lesssim\sum_{v\in C'}M_{HL}(|f\ast\psi_{2^{-j}r}\ast\sigma_r|^{\frac{2}{1+\epsilon}})(x+v).
	\end{align*}
	Therefore, we have
	\begin{align*}
		&\int\Big(\sup_{2^{-l}<r\leq2^{-l+1}}\sup_{v\in T}|f\ast\psi_{2^{-j}r}\ast\sigma_r(x+v)|\Big)^2\;dx\\
		&\lesssim\int\Big(\sup_{2^{-l}<r\leq2^{-l+1}}\sum_{v\in C'}M_{HL}(|f\ast\psi_{2^{-j}r}\ast\sigma_r|^{\frac{2}{1+\epsilon}})(x+v)\Big)^{1+\epsilon}\;dx\\
		&\lesssim2^{(j+l)s\epsilon}\sum_{v\in C'}\int \Big(M_{HL}(\sup_{2^{-l}<r\leq2^{-l+1}}|f\ast\psi_{2^{-j}r}\ast\sigma_r|^{\frac{2}{1+\epsilon}})(x+v)\Big)^{1+\epsilon}\;dx\\
		&\lesssim2^{(j+l)s\epsilon}\sum_{v\in C'}\int\Big(\sup_{2^{-l}<r\leq2^{-l+1}}|f\ast\psi_{2^{-j}r}\ast\sigma_r|(x)\Big)^{2}\;dx\\
		&\lesssim2^{(j+l)(1+\epsilon)s}\|f\ast\psi_{2^{-j}r}\ast\sigma_r\|_2^2,
	\end{align*}
	where we use a change of variable in $x$ and the boundedness of $M_{HL}$ in the second to last inequality.
	
	Let $\tilde{\psi}\in \mathcal S(\R^d)$ be such that $\widehat{\tilde{\psi}}(\xi)=1$ on $[\frac{1}{4},4]$ and $\supp(\widehat{\tilde{\psi}})\subset[\frac{1}{8},8]$. Then a scaling argument and Lemma 6.5.2 from \cite{classicalFAGrafakos} gives us that
	\begin{align*}
		\Big\|\sup_{2^{-l}<r\leq2^{-l+1}}|f\ast\psi_{2^{-j}r}\ast\sigma_r|\Big\|_2&=\Big\|\sup_{2^{-l}<r\leq2^{-l+1}}|f\ast\tilde{\psi}_{2^{-(j+l)}}\ast\psi_{2^{-j}r}\ast\sigma_r|\Big\|_2\\
		&\lesssim2^{-j(\frac{d-2}{2})}\|f\ast\tilde{\psi}_{2^{-(j+l)}}\|_2.
	\end{align*}
	Thus, the estimate above implies that
	\[\Big\|\sup_{v\in T}\sup_{2^{-l}<r\leq2^{-l+1}}|f\ast\psi_{2^{-j}r}\ast \sigma_{r}(x+v)|\Big\|_2\lesssim2^{\frac{(1+\epsilon)ls}{2}}2^{-\frac{(d-2-(1+\epsilon)s)j}{2}}\|f\ast\tilde{\psi}_{2^{-(j+l)}}\|_2.\]
	
	Now for $k\geq1$, using a scaling argument we obtain
	\[\Big\|\sup_{v\in T}\sup_{2^{-l}<r\leq2^{-l+1}}|f\ast\psi_{2^{k-j}r}\ast \sigma_{2^kr}(x+2^kv)|\Big\|_2\lesssim2^{\frac{(1+\epsilon)ls}{2}}2^{-\frac{(d-2-(1+\epsilon)s)j}{2}}\|f\ast\tilde{\psi}_{2^{k-(j+l)}}\|_2.\]
	Therefore, the Littlewood-Paley inequality and almost disjoint support of $\widehat{\tilde{\psi}_j}$'s will give us
	\[\Big\|\sup_{k\in\Z}\sup_{v\in T}\sup_{2^{-l}<r\leq2^{-l+1}}|f\ast\psi_{2^{k-j}r}\ast \sigma_{2^kr}(x+2^kv)|\Big\|_2\lesssim2^{\frac{(1+\epsilon)ls}{2}}2^{-\frac{(d-2-(1+\epsilon)s)j}{2}}\|f\|_2.\]	
	Interpolating the above $L^2$-estimates and $L^1$-estimate from \eqref{NjL1}, we get that for $1<p\leq2$
	\[\Big\|\sup_{k\in\Z}\sup_{v\in T}\sup_{2^{-l}<r\leq2^{-l+1}}|f\ast\psi_{2^{k-j}r}\ast \sigma_{2^kr}(x+2^kv)|\Big\|_p\lesssim2^{-j(d-s-1-\frac{d-s}{p})}2^{l(\frac{2d-(1+\epsilon)s}{p}-d+(1+\epsilon)s)}\|f\|_p.\]
	Hence, summing in $j$ concludes the proof of the lemma.
\end{proof}
We now complete the proof of the boundedness of $\mathcal{N}^T$ at points on the line segment $AQ$, excluding the point $Q$. By Lemma \ref{contlemma} and the weak type $(1,1)$ boundedness of $M_l$, we get that
\[\|\mathcal{N}^T(f_1,f_2)\|_{\frac{p}{p+1},\infty}\lesssim\sum_{l=1}^\infty 2^{-l(d-2)}\max\{2^{l(\frac{2d-(1+\epsilon)s}{p}-d+(1+\epsilon)s)},l2^{ls}\}\|f_1\|_1\|f_2\|_p\lesssim\|f_1\|_1\|f_2\|_p,\]
where the series in $l$ is summable for $p>1+\frac{1}{d-s-1}$ and $s<d-2$.
This completes the proof of Theorem \ref{bilinearNT}.\qed

\section{Proof of Theorem \ref{NT}: Lacunary uncentered spherical maximal operator 
$N^{T}_{lac}$}\label{Sec:prooflacuncentered}

Using the partition of unity ~\eqref{identity}, the lacunary maximal function $N_{lac}^T$ can be dominated by 
\begin{eqnarray}\label{reduction1}
	N^{T}_{lac}f(x)&\leq& \sup_{k\in\mathbb{Z}}|A^{T}_{2^k,0}f(x)|+\sum^{\infty}_{j=1}\sup_{k\in\mathbb{Z}}|A^{T}_{2^{k},j}f(x)|,
\end{eqnarray}
where   
\begin{align}
	A^{T}_{2^k,j}f(x)&=\sup_{u\in T}|(f*\psi_{2^{k-j}}*\sigma_{2^{k}})(x+2^ku)|,\;j\geq 1,\\
	A^{T}_{2^k,0}f(x)&=\sup_{u\in T}|(f*\phi_{2^k}*\sigma_{2^k})(x+2^ku)|.
\end{align}
We need to prove suitable estimates on the intermediary lacunary operators  
$$M^{T}_{j}f(x)=\sup_{k\in\mathbb{Z}}|A^{T}_{2^k,j}f(x)|,\;j\geq 0.$$
First, we observe that $M_0^Tf$ can be controlled by the classical Hardy-Littlewood maximal function $M_{HL}f$. 
Consider 
	\begin{eqnarray*}
		|A^{T}_{2^k,0}f(x)|&=&\sup_{u\in T}\left|\int_{\mathbb{S}^{d-1}}\phi_{2^k}*f(x+2^{k}(u+y))~d\sigma(y)\right|\\
		&\lesssim & \sup_{u\in T,|y|=1}|\phi_{2^k}*f(x+2^{k}(u+y))|\\
		&\lesssim & M_{HL}f(x).
	\end{eqnarray*}
Note that the final step in the above follows by the estimate~\eqref{prop1}. 

To deal with maximal operators $M_j^T$, we  will need the following $L^p$-estimates for the maximal averages $A^T_{1,j}$ from  
\cite{NikodymSetsAndMaximalFunctionsAssociatedWithSpheres}.
\begin{lemma}[\cite{NikodymSetsAndMaximalFunctionsAssociatedWithSpheres}]
	Let $j\in\N$. The following bounds hold true:
	\begin{align}
		\|A^{T}_{1,j}\|_{L^1(\R^d)\to L^1(\R^d)}&\lesssim \min\{2^{js},2^j\}. \label{L1-L1}\\
		\|A^{T}_{1,j}\|_{L^\frac{3}{2}(\R^d)\to L^\frac{3}{2}(\R^d)}&\lesssim j^\frac{1}{3}2^{\frac{j}{6}}, \label{L3/2-L3/2}\;\;d\geq3,\\
		\|A^{T}_{1,j}\|_{L^\frac{4}{3}(\R^d)\to L^\frac{4}{3}(\R^d)}&\lesssim j^\frac{1}{4}2^{\frac{j}{4}}, \label{L4/3-L4/3}\;\;d\geq4.
	\end{align}
\end{lemma}
The $L^1$-estimate in the above, is a consequence of Lemma \ref{coveringlemma} and the kernel estimate 
$$|\psi_{2^{-j}}*d\sigma(x)|\lesssim\frac{2^j}{(1+2^j||x|-1|)^N}.$$ The $L^p$-estimates at the points 
$p=\frac{3}{2},\frac{4}{3}$ were obtained in \cite{NikodymSetsAndMaximalFunctionsAssociatedWithSpheres} 
by proving almost sharp estimates for the Nikodym maximal function $N^\delta$ defined in 
Section \ref{Sec:Nikodym}.

Next, we illustrate a bootstrap argument which allows us to extend the range of $p$ for 
which $L^p$-boundedness of the multiscale operator $M_j^T$ holds, provided we have an initial 
$L^p$-estimate for the single scale operator $A_{1,j}^T$.
\begin{lemma}\label{vector}
	Let $1\leq p_1,p_2\leq 2$ be such that 
	\begin{align*}
		\|A_{1,j}^T\|_{L^{p_1}\to L^{p_1}}&\leq C_1,\\
		\|M_j^T\|_{L^{p_2}\to L^{p_2}}&\leq C_2.
	\end{align*}
	Then, we have
	\[\|M_j^T\|_{L^{p}\to L^{p}}\lesssim C_1^\frac{p_1}{2} C_2^{1-\frac{p_1}{2}},\;\;\text{for}\;p=\frac{2p_2}{2+p_2-p_1}.\]
\end{lemma}
\begin{proof}
	The proof involves a vector-valued argument. Consider the following vector-valued operator 
	$\boldsymbol{\vec A}$ acting on a sequence of measurable functions $f=\{f_{k}\}_{k\in\Z}$.
	\[\boldsymbol{\vec A}(\{f_k\}_{k\in\Z})(x)=\{A_{2^k,j}^T(f_{k})(x)\}_{k\in\Z}.\]
	First, observe that the operator $A_{1,j}^T$ readily extends to a vector-valued setting, namely we get that 
	\begin{equation}\label{Lp}
		\|\boldsymbol{\vec A}\|_{L^{p_1}(\ell_{p_1})\to L^{p_1}(\ell_{p_1})}\leq C_1.
	\end{equation}
	Next, using the $L^{p_2}$-boundedness of $M_j^{T},$ we get that
	\begin{eqnarray*}
		\|\boldsymbol{\vec A}f\|_{L^{p_2}(\ell_{\infty})}&=&\|\sup_k A^{T}_{2^k,j}(f_{k})\|_{{p_2}}\\
		&\leq& \|M_j^{T}(\sup_{m}|f_{m}|)\|_{{p_2}}\\
		&\leq &C_2\|f\|_{L^{p_2}(\ell_{\infty})}.
	\end{eqnarray*}
	Therefore, we have that 
	\begin{equation}\label{Lq}
		\|\boldsymbol{\vec A}\|_{L^{p_2}(\ell_{\infty})\to L^{p_2}(\ell_{\infty})}\leq C_2.
	\end{equation}
	Interpolate between \eqref{Lp} and \eqref{Lq} to deduce that 
	\begin{equation}\label{interpo}
		\|\boldsymbol{\vec A}\|_{L^{p}(\ell_{2})\to L^{p}(\ell_{2})}\lesssim C_1^\frac{p_1}{2} C_2^{1-\frac{p_1}{2}}.
	\end{equation}
	This estimate can be used to get that 
	\begin{eqnarray*}
		\|M^{T}_{j}f\|_{p}&\leq& \left\| \left(\sum_{k\in\mathbb{Z}}|A^{T}_{2^k,j}f|^2\right)^{\frac{1}{2}}\right\|_{p}\\
		&\lesssim& C_1^\frac{p_1}{2} C_2^{1-\frac{p_1}{2}}\left\| \left(\sum_{k\in\mathbb{Z}}|\psi_{2^{k-j}}*f|^2\right)^{\frac{1}{2}}\right\|_{p}\\
		&\lesssim&C_1^\frac{p_1}{2} C_2^{1-\frac{p_1}{2}} \|f\|_{p},
	\end{eqnarray*}
	where in the last step, we have used the Littlewood-Paley inequality (see for example, \cite{Duoandikoetxeabook}).
\end{proof}

We now state the main estimates for the operators $M_j^T$ needed to prove Theorem \ref{NT}.
\begin{lemma}Let $j\in\N$ and $0\leq s<d-1$. We have the following estimates,\label{Mjestimates}
	\begin{enumerate}
		\item For $d\geq2$,
		\begin{equation}\label{l2estimate}
			\Vert M_{j}^{T}f\Vert_{2}\lesssim 2^{-\frac{(d-1-s)j}{2}}\Vert f\Vert_{2}.
		\end{equation}
		\item For $d\geq 2,\;p_0=1+\frac{s}{2(d-1)}$,
		\begin{equation}\label{Ms/2}
			\|M_j^T\|_{L^{p_0}(\R^d)\to L^{p_0}(\R^d)}\lesssim 2^{j\frac{s(d-1)}{2(d-1)+s}}.
		\end{equation}
		\item For $d=3$, $\frac{3}{2}<p\leq2$,
		\begin{equation}\label{M3/2}
			\|M_j^T\|_{L^p(\R^3)\to L^p(\R^3)}\lesssim j^{\frac{1}{p}-\frac{1}{2}}2^{-j\left(\frac{9-4s}{2}-\frac{7-3s}{p}\right)}.
		\end{equation}
		\item For $d\geq4,\;\frac{4}{3}<p\leq2$,
		\begin{equation}\label{M4/3}
			\|M_j^T\|_{L^p(\R^d)\to L^p(\R^d)}\lesssim j^{\frac{1}{p}-\frac{1}{2}}2^{-j\left(\frac{3d-2-3s}{2}-\frac{2d-1-2s}{p}\right)}.
		\end{equation}
	\end{enumerate}
\end{lemma}
\begin{proof}
	\textbf{Proof of (1):} By an application of Lemma \ref{coveringlemma}, we get that 
	\begin{eqnarray*}
		\|A^{T}_{1,j}f\|_2&=&\|\sup_{u\in T}|f\ast\psi_{2^{-j}}\ast\sigma(\cdot+u)|\|_{2}\\
		&\lesssim& 2^{\frac{js}{2}}\|f\ast\psi_{2^{-j}}\ast\sigma\|_{2}.
	\end{eqnarray*}
	We use Plancherel's  identity and standard scaling argument to get that  
	\begin{align}
		\Vert A^{T}_{2^k,j}f\Vert_{2}\lesssim 2^{-\frac{(d-1-s)j}{2}}\Vert f\Vert_{2}.
	\end{align}
	Let $\tilde{\psi}\in \mathcal S(\R^d)$ be such that $\widehat{\tilde{\psi}}(\xi)=1$ on the support of $\hat{\psi}$. Since, $\widehat{\tilde{\psi}}_{2^k}\widehat{{\psi}}_{2^k}=\widehat{{\psi}}_{2^k}$ for all $k,$ we can use the orthogonality of Fourier transform in the following way to get the desired result.  
	\begin{eqnarray*}
		\Vert M^{T}_{j}f\Vert^{2}_{2}&\leq& \left\| \left(\sum_{k\in\mathbb{Z}}|A^{T}_{2^k,j}f|^2\right)^\frac{1}{2}\right\|^{2}_{2}= \left\|\left( \sum_{k\in\mathbb{Z}}|A^{T}_{2^k,j}(\tilde{\psi}_{2^{k-j}}*f)|^2\right)^\frac{1}{2}\right\|^{2}_{2}\\
		&\lesssim& 2^{-(d-1-s)j}\sum_{k}\Vert \tilde{\psi}_{2^{k-j}}*f\Vert^{2}_{2}\\
		&\lesssim & 2^{-(d-1-s)j}\Vert f\Vert^{2}_{2}.
	\end{eqnarray*}
	\textbf{Proof of (2):} By an application of Lemma \ref{vector} along with the estimates \eqref{L1-L1} and \ref{l2estimate}, we have that
	\[\|M_j^T\|_{L^\frac{4}{3}\to L^\frac{4}{3}}\lesssim 2^{-j\frac{(d-1-3s)}{4}}.\]
	By a recursive application of Lemma \ref{vector} along with \eqref{L1-L1} and the above estimate, we obtain 
	the estimate \ref{Ms/2}.\\
	\textbf{Proof of (3):} The inequality \ref{M3/2} follows from repeated application of Lemma \ref{vector} 
	along with the estimates \eqref{L3/2-L3/2} and \ref{l2estimate}.\\
	\textbf{Proof of (4):} The proof is similar to that of inequality \ref{M3/2} with the exception of using 
	estimate \ref{L4/3-L4/3} instead of \ref{L3/2-L3/2}.
\end{proof}
To obtain a restricted weak type inequality, we will employ an interpolation argument
due to Bourgain. 
We state the lemma for convenience. The interested reader is referred 
to \cite{Lee1} (Lemma 2.6) for details. 
\begin{lemma}[\cite{Lee1}]\label{Bourgain}
	Let $\epsilon_1,\epsilon_2>0$. Suppose that $\{T_j\}$ is a sequence of linear (or sublinear) operators 
	such that for some $1\leq p_1,p_2<\infty$, and $1\leq q_1,q_2<\infty$, 
	$$\Vert T_{j}(f)\Vert_{L^{q_1}}\leq M_12^{\epsilon_1 j}\Vert f\Vert_{L^{p_1}},~~\Vert T_{j}(f)\Vert_{L^{q_2}}\leq M_22^{-\epsilon_2 j}\Vert f\Vert_{L^{p_2}}.$$
	Then $T=\sum_jT_j$ is bounded from $L^{p,1}$ to $L^{q,\infty}$, i.e. 
	$$\Vert T(f)\Vert_{L^{q,\infty}}\lesssim M^{\theta}_{1}M^{1-\theta}_{2}\Vert f\Vert_{L^{p,1}},$$	
	where $\theta=\frac{\epsilon_2}{\epsilon_1+\epsilon_2}$, $\frac{1}{q}=\frac{\theta}{q_1}+\frac{1-\theta}{q_2}$ 
	and $\frac{1}{p}=\frac{\theta}{p_1}+\frac{1-\theta}{p_2}$.
\end{lemma}
Since we have all the ingredients, we now conclude the proof of Theorem \ref{NT}.
\begin{proof}[Proof of Theorem \ref{NT}]
	First, we prove the restricted weak type inequality at the endpoint $p_0=1+\frac{s}{d-1}$.
	
	The restricted weak type inequality $N_{lac}^T:L^{1+\frac{s}{d-1},1}\to L^{1+\frac{s}{d-1},\infty}$ follows 
	from the estimates \ref{l2estimate}, \eqref{Ms/2} along with an application of Lemma \ref{Bourgain} for 
	the operators $M_j^T$. 
	
	The proof of restricted weak type estimate for the endpoint $p=1+\frac{1}{d-s}$ is simpler. Indeed, it 
	follows by applying Lemma \ref{Bourgain} to the endpoint estimates \ref{l2estimate} and $\|M_j^T\|_{L^1\to L^1}\lesssim 2^j$.
	
	The $L^p$-estimates for $N^T_{lac}$ for the range 
	$p>1+\min\{\frac{s}{d-1},\frac{1}{d-s}\}$ follow 
	from the interpolation of respective restricted weak type inequalities and the trivial $L^\infty$-estimate. 
	Finally, we obtain the better $L^p$-bounds for the case $p>1+\frac{5-2s}{9-4s}$ and $p>1+\frac{d-s}{3(d-s)-2}$ 
	in dimensions $d=3$ and $d\geq4$ respectively by resorting to inequalities \ref{M3/2} and \ref{M4/3} for 
	$M_j^T$ and summing in $j$.
\end{proof}
\section{Proof of Theorem \ref{Lp improving}: $L^p$-improving properties of $A^T$}\label{Sec:ProofLp}
In view of the real interpolation theory, it is enough to establish restricted 
weak-type estimates for the operator 
$A^T$ at the endpoints described in Theorem~\ref{Lp improving}. 
This is obtained by decomposing the operator 
$A^T$ at dyadic scales and by performing a discretization of the set $T$ adapted 
to each dyadic scale. We will prove 
suitable $L^p$-estimates for each piece of $A^T$. Finally, using the 
interpolation theorem due to Bourgain we get the desired restricted weak-type estimates. 

Let $\phi$ and $\psi$ be as in~\ref{identity}. This gives us the following decomposition,
\[A^{T}_{1}f(x)\leq\sum\limits_{j=0}^\infty A^{T}_{1,j}f(x).\]
Next, consider the covering of $T$ by balls of radius $\delta=2^{-j}.$ Invoking the covering argument 
from Lemma \ref{coveringlemma}, we get that 
$$\Vert A^{T}_{1,j}|\Vert^{p}_{L^p}\lesssim N(T,2^{-j}) \Vert f\ast\psi_{2^{-j}}\ast\sigma\Vert^p_{L^{p}},~p\geq 1.$$
Recall that we have the bound $N(T,2^{-j})\leq 2^{js}.$ So by using the Fourier transform estimate 
\begin{align}
	\|A^{T}_{1,j}\|_{L^2\to L^2}&\lesssim 2^{-j\left(\frac{d-s-1}{2}\right)}. \label{L2-L2}
\end{align}
Moreover, using the $L^p$-improving estimates of $A^{u}_1$ \cite{Littman} (for a fixed $u\in T$) and applying the covering argument we get
\begin{equation}\label{Lp-Lq}
	\|A^{T}_{1,j}\|_{L^{\frac{d+1}{d}}\to L^{d+1}}\lesssim 2^{j\left(\frac{s}{d+1}\right)}.
\end{equation}

Now using the estimates \eqref{L2-L2}, \eqref{L1-L1}, and \eqref{Lp-Lq}, and Lemma \eqref{Bourgain} we get 

\[\|A^{T}_{1}f\|_{L^{q,\infty}}\lesssim \|f\|_{L^{p,1}},\]
where $(1/p,1/q)= H$ and $E$. This completes the proof. \qed

\subsection{Necessary conditions}\label{Sec:nec}
In this section, we construct an example to obtain sharpness of the restricted weak type inequality $A_1^T:L^{1+\frac{s}{d-1},r}\to L^{1+\frac{s}{d-1},\infty}$. The example is a modification of the arguments given in \cite{EndpointMappingPropertiesOfSphericalMaximalOperators}, where they showed the sharpness in the particular case when $T$ is of Minkowski dimension $\frac{1}{2}$ in dimension $d=2$.

We consider two self-similar sets of appropriate Minkowski content whose sum is an interval. We say $C$ is a Cantor set of ratio $\lambda$ if $C$ is self-similar with respect to the similitudes $S_1(x)=\lambda x$ and $S_2(x)=\lambda x-(1-\lambda)$ i.e.,  $C=S_1(C)\cup S_2(C)$. It is well-known that the Minkowski dimension of $C$ is $\frac{\log 2}{\log \frac{1}{\lambda}}$, see Section 2.2 of Chapter 7 in \cite{SSbook}. The following lemma justifies the existence of two Cantor sets whose sum is an interval.
\begin{lemma}\label{sum}
	Let $C_1$ and $C_2$ be two Cantor sets with ratios $\lambda_1$ and $\lambda_2$ respectively such that $\frac{\lambda_1}{1-2\lambda_1}\frac{\lambda_2}{1-2\lambda_2}\geq 1$. Then $C_1+C_2$ is a closed interval. In fact, we have $C_1+C_2=[0,2]$.
\end{lemma}
We refer the interested reader to [Theorem 1.1, \cite{Takahashi}] for more details. 
\begin{prop}
	For $r>1$, the operator $A_1^T$ does not map $L^{1+\frac{s}{d-1},r}$ to  $L^{1+\frac{s}{d-1},\infty}$.
\end{prop}
\begin{proof}
	Let $C_1,C_2\subset[0,1]$ be Cantor sets of ratios $2^{-\frac{1}{s}}$ and $2^{-\frac{1}{1-s}}$ respectively. Then from Lemma \ref{sum} we can see that $[-1,1]= C_2-C_1=\{c_2-c_1:c_1\in C_1,c_2\in C_2\}$. For a small $a>0$, define 
	\[f(x)=\sum_{i=1}^{N}4^{(d-1)i}\chi_{\{ce_1:c\in C_2\}+B(0,a4^{-i})}(x).\]
	For $r>1$, we have the following bound on Lorentz space norm of $f$,
	\[\|f\|_{L^{1+\frac{s}{d-1},r}}\lesssim N^{\frac{1}{r}}.\]
	To see this, consider the sets $F_j=\bigg(4^{(d-1)j}\sum\limits_{k=0}^{j-1}4^{-(d-1)k},4^{(d-1)(j+1)}\sum\limits_{k=0}^{j}4^{-(d-1)k}\bigg]$ for $1\leq j\leq N-1$. If $t\in F_j$ , we have 
	\[\{x\in\R^d:\;|f(x)|>t\}=\bigcup\limits_{y\in C}y+B(0,a4^{-(j+1)}).\]
	Denote $d_f(t)=|\{x\in\R^d:\;|f(x)|>t\}|$ and observe that, $$td_f(t)^\frac{1}{p}=t[(a4^{-(j+1)})^{d-(1-s)}]^\frac{d-1}{d-1+s}\lesssim1.$$
	Therefore, we have 
	\begin{align*}
		\|f\|_{L^{1+\frac{s}{d-1},r}}&= \left(\frac{d-1+s}{d-1}\right)^\frac{1}{r}\left(\int_0^{4^{d-1}}[d_f(t)t]^r\frac{dt}{t}+\sum_{j=0}^{N-1}\int_{F_j}[d_f(t)t]^r\frac{dt}{t}\right)^\frac{1}{r}\\
		&\lesssim \left(\int_0^{4^{d-1}}4^{-(d-1)}t^{r-1}\;dt+\sum_{j=0}^{N-1}\int_{F_j}\frac{dt}{t}\right)^\frac{1}{r}\\
		&\lesssim \left(4^{(r-1)(d-1)}+\sum_{j=0}^{N-1}1\right)^\frac{1}{r}\lesssim N^\frac{1}{r}
	\end{align*}
	
	Let $T=\{ce_1:\;c\in C_1\}$. Now for $x\in [-1,1]e_1+\mathbb{S}^{d-1}$, there exists $u\in T$ and $y\in\mathbb{S}^{d-1}$ such that
	\[x+u+y=z\in \{ce_1:c\in C_2\} \text{\;\;and\;\;} |\mathbb{S}^{d-1}(x+u)\cap B(z,a4^{-i})|\simeq 4^{-i(d-1)}.\]
	Thus, we can see that $A^T_1f(x)\gtrsim N$. Suppose $A^T_1$ maps $L^{1+\frac{s}{d-1},r}$ to $L^{1+\frac{s}{d-1}}$, then
	\begin{eqnarray*}
		1\lesssim|\{x:A^T_1f(x)\gtrsim N\}|&\lesssim& \left(\frac{\|A^T_1\|_{L^{1+\frac{s}{d-1},r}\to L^{1+\frac{s}{d-1},\infty}}\|f\|_{L^{1+\frac{s}{d-1},r}}}{N}\right)^{1+\frac{s}{d-1}}\\
		&\lesssim& \frac{\|A^T_1\|_{L^{1+\frac{s}{d-1},r}\to L^{1+\frac{s}{d-1},\infty}}}{N^{\left(1-\frac{1}{r}\right)\left(1+\frac{s}{d-1}\right)}}.
	\end{eqnarray*}
	The desired result follows by taking the limit $N\to\infty$.
\end{proof}

\section{Proof of Theorems \ref{Nikodymlacunary} and \ref{Ndelta}: Nikodym maximal operators}\label{Sec:proofNikodym}
\subsection{Proof of Theorem \ref{Nikodymlacunary}:} Theorem~\ref{Nikodymlacunary} follows by using similar 
arguments as in the case Theorem \ref{NT}. We give a brief sketch of the proof. 

Let us first, recall the $L^p$-estimates for the single scale operator $N^{\delta}_{1}$ 
from  [Theorem 1.3, \cite{NikodymSetsAndMaximalFunctionsAssociatedWithSpheres}]. 
\begin{align}
	\|N^{\delta}_{1}\|_{L^1(\R^d)\to L^1(\R^d)}&\lesssim \delta^{-1}\label{L1toL1}\\
	\|N^{\delta}_{1}\|_{L^\frac{3}{2}(\R^d)\to L^\frac{3}{2}(\R^d)}&\lesssim \delta^{-\frac{1}{6}}(-\log\delta)^\frac{1}{3}\label{L3/2toL3/2}\\
	\|N^{\delta}_{1}\|_{L^\frac{4}{3}(\R^d)\to L^\frac{4}{3}(\R^d)}&\lesssim \delta^{-\frac{1}{4}}(-\log\delta)^\frac{1}{4}.\label{L4/3toL4/3}
\end{align}
Denote 
$$N^{\delta}_{k}f(x)=\sup_{u\in \mathbb{S}^{d-1}}\frac{1}{|S^{\delta}(0)|}\Big|\int_{S^{\delta}(0)}f(x+2^{k}(u+y))dy\Big|.$$
As in the proof of Theorem~\ref{Nikodymlacunary}, using the identity \eqref{identity}, we get the following decomposition of $N^{\delta}_{k}f.$
$$N^{\delta}_{k}f(x)=N^{\delta}_{k}(\phi_{2^k}*f)(x)+\sum^{\infty}_{j=1}N^{\delta}_{k}(\psi_{2^{k-j}}*f)(x).$$
Consequently, we get that 
\begin{eqnarray}\label{reduction2}
	N^{\delta}_{lac}f(x)&\leq& \sup_{k\in\mathbb{Z}}|N^{\delta}_{k}(\phi_{2^k}*f)(x)|+\sum^{\infty}_{j=1}\sup_{k\in\mathbb{Z}}|N^{\delta}_{k}(\psi_{2^{k-j}}*f)(x)|.
\end{eqnarray}
Let $M_j^\delta$ denote the intermediary maximal operator 
$$M^{\delta}_{j}f(x)=\sup_{k\in\mathbb{Z}}|N^{\delta}_{k}(\psi_{2^{k-j}}*f)(x)|.$$
We observe that the arguments used to deal $M_0^T$ apply here as well and we get the corresponding estimate for $M_0^\delta f,$ namely,   
\[M_0^\delta f(x)\lesssim M_{HL}f(x),~\text{a.e}~x\in\mathbb{R}^{d}.\]
Next, we have the following $L^2$-estimate for the operators $M_j^\delta$.
\begin{lemma}\label{l2delta}
	For $d\geq2$ and $j\in \mathbb{N},$ the following holds
	$$\Vert M_{j}^{\delta}f\Vert_{2}\lesssim \delta^{-\epsilon}2^{-j\epsilon}\Vert f\Vert_{2}.$$
	Consequently, we get  that 
	$$\Vert N^{\delta}_{lac} \Vert_{2}\lesssim \delta^{-\epsilon}\Vert f\Vert_{2}.$$   
\end{lemma}
\begin{proof}
	We claim that for any $0<\epsilon<1$ the following estimate holds,
	\begin{equation}\label{Fourierest}
		\frac{|\widehat{\chi_{S^{\delta}(0)}}(\xi)|}{|S^{\delta}(0)|}\lesssim \frac{\delta^{-\epsilon}}{(1+|\xi|)^{\frac{d-1}{2}+\epsilon}}.
	\end{equation}
	Indeed, to prove the above inequality it is enough to show
	\[|\widehat{\chi_{S^{\delta}(0)}}(\xi)|\lesssim\min\left\{\frac{1}{(1+|\xi|)^{\frac{d+1}{2}}},\frac{\delta}{(1+|\xi|)^{\frac{d-1}{2}}}\right\}.\]
	These estimates can be obtained by using the following identity and mean value theorem,
	\begin{eqnarray*}
		\widehat{\chi_{S^{\delta}(0)}}(\xi)&=&\frac{(1+\delta)^{\frac{d}{2}}J_{\frac{d}{2}}((1+\delta)|\xi|)}{|\xi|^\frac{d}{2}}-\frac{(1-\delta)^{\frac{d}{2}}J_{\frac{d}{2}}((1-\delta)|\xi|)}{|\xi|^\frac{d}{2}}.
	\end{eqnarray*}
	From the covering argument of Lemma \ref{coveringlemma} and \eqref{Fourierest}, we obtain that
	\begin{eqnarray}
		\|N^{\delta}_1(\psi_{2^{-j}}*f)\|_2&=& \left\|\sup_{u\in\mathbb{S}^{d-1}}\left|f\ast\psi_{2^{-j}}\ast\frac{\chi_{S^{\delta}(0)}}{|S^{\delta}(0)|}(\cdot+u)\right|\right\|_2\nonumber\\
		&\lesssim& 2^{j\frac{d-1}{2}}\left\|f\ast\psi_{2^{-j}}\ast\frac{\chi_{S^{\delta}(0)}}{|S^{\delta}(0)|}(\cdot)\right\|_2\nonumber\\
		&\lesssim& 2^{j\frac{d-1}{2}}\frac{\delta^{-\epsilon}}{2^{j\left(\frac{d-1}{2}+\epsilon\right)}}\|f\|_2\nonumber\\
		&=& \delta^{-\epsilon}2^{-j\epsilon}\|f\|_2\label{L2delta}.
	\end{eqnarray}
	By an scaling argument and the estimate \eqref{L2delta}, we have that 
	\begin{eqnarray*}
		\Vert M^{\delta}_{j}f\Vert^{2}_{2}&\leq& \left\| \left(\sum_{k\in\mathbb{Z}}|N^{\delta}_{k}(\psi_{2^{k-j}}*f)|^2\right)^\frac{1}{2}\right\|^{2}_{2}= \left\| \left(\sum_{k\in\mathbb{Z}}|N^{\delta}_{k}(\psi_{2^{k-j}}* \widetilde\psi_{2^{k-j}}*f)|^2\right)^\frac{1}{2}\right\|^{2}_{2}\\
		&\lesssim& \delta^{-2\epsilon}2^{-2j\epsilon}\sum_{k}\Vert \widetilde\psi_{2^{k-j}}*f\Vert^{2}_{2}\lesssim \delta^{-2\epsilon}2^{-2j\epsilon}\Vert f\Vert^{2}_{2}.
	\end{eqnarray*}
	Here, in the last step we have used the Littlewood-Paley inequality.
\end{proof}

By a bootstrapping argument similar to that in the proof of Theorem \ref{NT} with estimates \eqref{L1toL1}, \eqref{L3/2toL3/2}, \eqref{L4/3toL4/3} and Lemma \ref{l2delta}, we obtain the required estimates for $N^\delta_{lac}$.

A recursive application of the arguments similar to that of Lemma \ref{vector} for $M^\delta_j$ along with Lemma \ref{l2delta} and the estimates \eqref{L3/2toL3/2}, \eqref{L4/3toL4/3} gives us that
\begin{align}
	\|M_j^\delta\|_{L^p\to L^p}&\lesssim 2^{-j\epsilon\left(4-\frac{6}{p}\right)}\delta^{-\epsilon}\delta^{\frac{1}{p}-\frac{1}{2}}, &\text{for}\;d\geq3,\;\frac{3}{2}<p\leq2,\\
	\|M_j^\delta\|_{L^p\to L^p}&\lesssim 2^{-j\epsilon\left(3-\frac{4}{p}\right)}\delta^{-\epsilon}\delta^{\frac{1}{p}-\frac{1}{2}},&\text{for}\;d\geq4,\;\frac{4}{3}<p\leq2.    
\end{align}
\qed
\subsection{Proof of Theorem \ref{Ndelta}:}\label{Sec:proofbilinearNikodym}
The following result will be useful in computing the $L^\frac{1}{2}$-norm of the operator using appropriate $L^1$-bounds.  
\begin{lemma}[\cite{IPS}]\label{Sovine}
	Let $T$ be a bilinear operator satisfying the following. 
	\begin{enumerate}[label=(\roman*)]
		\item There exists $N>0$ such that for functions $f_i,f_j$ supported in the unit cubes $Q_{l_i}$ and $Q_{l_j}$ with their lower left corners at $l_i$ and $l_j$ respectively and $\|l_i-l_j\|_\infty>N,$ we have $T(f_i,f_j)=0.$
		\item There exists $R>0$ such that $T(f_1,f_2)$ is supported on $(\mathrm{supp }(f_1)+B(0,R))\bigcup(\mathrm{supp }(f_2)+B(0,R)).$
		\item The operator $T$ satisfies the $L^p$-estimate
		\[\|T(f_1,f_2)\|_{p_3}\leq C\|f_1\|_{p_1}\|f_2\|_{p_2},\]
		for all functions $f_1,f_2$ supported in a fixed cube and for some exponents $p_1, p_2$ and $p_3$ with $p_1,p_2\geq 1, \frac{1}{p_1}+\frac{1}{p_2}=\frac{1}{p}>1$ and $p_3\geq p$.  
	\end{enumerate}
	Then, for all $r\in[p,p_3],$ we have 
	\[\|T\|_{L^{p_1}\times L^{p_2}\to L^r}\lesssim C.\]
\end{lemma}
Now we proceed with the proof of Theorem \ref{Ndelta}.
First, we prove the boundedness of $\mathcal N^\delta$ in the Banach range. We observe the trivial $L^\infty$ estimate
\begin{eqnarray}\label{Linfforbil} 
	\|\mathcal N^\delta\|_{L^\infty\times L^\infty\to L^\infty}\leq 1.
\end{eqnarray}
The H\"older's inequality yields 
\[\mathcal N^\delta(f_1,f_2)(x)\leq N_pf_1(x)N_{p'}f_2(x),\;\;\text{for }\frac{1}{p}+\frac{1}{p'}=1,\;p\geq1,\]
where $N_pf(x)=\sup\limits_{u\in\mathbb{S}^{d-1}}\left(\frac{1}{\delta}\int_{S^\delta(0)}|f(x+u+y)|^p\;d(y,z)\right)^\frac{1}{p}.$

We use a change of variable to polar co-ordinates and the well-known slicing technique, see~ \cite{MaximalEstimatesForTheBilinearSphericalAveragesAndTheBilinearBochnerRieszOperators}, to prove the following estimates for $N_pf.$ Consider 
\begin{align*}
	\|N_pf\|_p^p&\leq\int_{\R^d}\sup\limits_{u\in\mathbb{S}^{d-1}}\frac{1}{\delta}\int_{t=1-\delta}^{1+\delta}\int_{\mathbb{S}^{2d-1}}|f(x+u+ty)|^p\;d\sigma_{2d}(y,z)t^{2d-1}dtdx\\
	&\lesssim\int_{\R^d}\sup\limits_{u\in\mathbb{S}^{d-1}}\frac{1}{\delta}\int_{t=1-\delta}^{1+\delta}\int_{B_d(0,1)}|f(x+u+ty)|^p\;(1-|y|^2)^{\frac{d-2}{2}}dydtdx\\
	&\lesssim\int_{\R^d}\sup\limits_{u\in\mathbb{S}^{d-1}}\frac{1}{\delta}\int_{t=1-\delta}^{1+\delta}\int_{B_d(u,t)}|f(x+y)|^p\;dyt^{-d}dtdx\\
	&\lesssim\int_{\R^d}\frac{1}{\delta}\int_{t=1-\delta}^{1+\delta}\int_{B_d(0,10)}|f(x+y)|^p\;dydtdx\\
	&\leq\frac{1}{\delta}\int_{t=1-\delta}^{1+\delta}\int_{B_d(0,10)}\|f\|_p^p\;dydt\\
	&\lesssim\|f\|_p^p
\end{align*}
The estimate above implies that $\|\mathcal N^\delta(f_1,f_2)\|_1\lesssim\|f_1\|_p\|f_2\|_{p'}.$ The desired result for other exponents in the Banach triangle follows by interpolation.

Next, we prove the $L^1\times L^1\to L^{\frac{1}{2}}$ bound of $\mathcal N^\delta$. We can see that
\begin{align*}
	\mathcal N^\delta (f_1,f_2)(x)&\leq \frac{1}{\delta}\left|\int_{B_{2d}(0,10)} f_1(x+y)f_2(x+z)\;d(y,z)\right|\\
	&\leq \frac{1}{\delta}\int_{B_{d}(0,10)} |f_1(x+y)|\;dy \int_{B_{d}(0,10)}|f_2(x+z)|\;dz\\
	&\leq \frac{1}{\delta}\|f_1\|_1 \int_{B_{d}(0,10)}|f_2(x+z)|\;dz.
\end{align*}
From the above estimates, we obtain 
\[\|\mathcal N^\delta (f_1,f_2)\|_1\lesssim \frac{1}{\delta}\|f_1\|_1\|f_2\|_1.\]
The required $L^1(\mathbb R^d)\times L^1(\mathbb R^d)\to L^\frac{1}{2}(\mathbb R^d)$-estimate follows from an application of Lemma \ref{Sovine}.
\qed

\section{Boundedness of $\mathcal{N}^{T}_{lac}$ for the case $d=1$}\label{onedimensional}
	In dimension $d=1$, one can obtain $L^{p_1}\times L^{p_2}\to L^{p_3}$-boundedness for $1< p_1,p_2\leq\infty$, $1/p_1+1/p_2=1/p_3$ and $p_3>1$ when $s<\frac{1}{2}$, where $s$ is the upper Minkowski content of $T\subset\mathbb{R}$. Indeed, we have
	\begin{align*}
		\mathcal{N}^{T}_{lac}(f_1,f_2)(x)\leq \Vert f_2\Vert_{\infty} \sup_{u\in T,k\in\mathbb{Z}}\Big|\int_{\mathbb{S}^{1}}f_1(x+2^{k}(u+y))~d\sigma(y,z)\Big|.
	\end{align*}
	By symmetry and a change of variable, it is enough to consider
	\begin{align*}
		\sup_{u\in T,k\in \mathbb{Z}}\int^{1}_{0}|f_1(x+2^k(u+t))|~\frac{dt}{\sqrt{1-t^2}}&= \sup_{u\in T,k\in \mathbb{Z}}\sum_{j=1}^\infty 2^{\frac{j}{2}}\int_{1-2^{-(j-1)}}^{1-2^{-j}}|f_1(x+2^k(u+t))|~dt\\
		&=\sup_{u\in T,k\in \mathbb{Z}}\sum_{j=1}^\infty \frac{2^{-\frac{j}{2}}}{|I_j|}\int_{I_j+u}|f_1(x+2^kt)|~dt,
	\end{align*}
	where $I_j=[1-2^{-(j-1)},1-2^{-j}]$. Let $\{{u}^{0}_{j,l}\}_{l\in\Lambda}$ be the centers of intervals $\{\mathcal{C}_{j,l}\}_{l\in\Lambda}$ of length $2^{-j-1}$ covering $T$ such that $\#\{{u}^{0}_{j,l}\}\lesssim 2^{js}$. Then for  any $u\in T\cap[u^{0}_{j,l}-2^{-(j+1)},u^{0}_{j,l}+2^{-(j+1)}]$, we have
	\[\int_{I_j+u}|f_1(x+2^kt)|~dt\leq\int_{\tilde{I_j}+u_0}|f_1(x+2^kt)|~dt,\]
	where $\tilde{I_j}=[1-5\cdot2^{-(j+1)},1-2^{-(j+1)}]$ and $|\tilde{I_j}|=4|I_j|$. Thus, for $p>1$, we have
	\begin{align*}
		\left\|\sup_{u\in T,k\in \mathbb{Z}}\int^{1}_{0}|f_1(x+2^k(u+t))|~\frac{dt}{\sqrt{1-t^2}}\right\|_p&\lesssim\left\|\sum_{j=1}^{\infty}2^{-\frac{j}{2}}\sum_{u^{0}_{j,l}\in\mathcal{C}_{j,l}}\sup_{k\in\Z}\frac{1}{|\tilde{I_j}|}\int_{\tilde{I_j}+u^{0}_{j,l}}|f_1(x+2^kt)|~dt\right\|_p\\
		&\leq \sum_{j=1}^{\infty}2^{-\frac{j}{2}}\sum_{u^{0}_{j,l}\in\mathcal{C}_{j,l}}\left\|\sup_{k\in\Z}\frac{1}{|\tilde{I_j}|}\int_{\tilde{I_j}+u^{0}_{j,l}}|f_1(x+2^kt)|~dt\right\|_p\\
		&\lesssim\sum_{j=1}^{\infty}2^{js-\frac{j}{2}}\log(2+c2^{j})\|f_1\|_p,
	\end{align*}
	where we have used the $L^p$-estimates of Hardy-Littlewood maximal function. We note that sum in the above inequality is finite if $s<\frac{1}{2}$. Now, interchanging the role of $f_1$ and $f_2$, we can obtain $L^{p_1}\times L^{p_2}\to L^{p_3}$-boundedness of $\mathcal{N}^{T}_{lac}$ for $p_3>1$ when $s<\frac{1}{2}$.
	
	Further, we can improve the range of boundedness for exponents $(p_1,p_2,p_3)$ beyond Banach range for 
	$\mathcal{N}^{T}_{lac}$ using the decay in $L^2\times L^2\to L^1$-estimate when Minkowski content of 
	$T$ is small. Indeed, for some fixed $u,v\in T$, we rewrite
	\begin{eqnarray*}
		\mathcal{A}^{u,v}_{1}(f_1,f_2)(x)&=&\int_{\mathbb{S}^1}f_1(x+u+y)f_2(x+v+z)~d\sigma(y,z)\\
		&=&\int^{2\pi}_{0}f_1(x+u+\cos\theta)f_2(x+v+\sin\theta)~d\theta\\
		&:=&\int^{2\pi}_{0}f_1(x+\gamma_1(\theta))f_2(x+\gamma_2(\theta))~d\theta.
	\end{eqnarray*}
	We observe that the curve $\gamma(\theta)=(\gamma_1(\theta),\gamma_2(\theta))$, $\theta\in [0,2\pi]$ 
	satisfies all the three hypothesis of the class of singular curves defined in \cite{ChristZhou}. 
	Therefore, by Theorem $2.2$ of \cite{ChristZhou}  we get that there exists $\sigma<0$ such that
	\[\Vert \mathcal{A}^{u,v}_{1}(f_1,f_2)\Vert_{L^{1}}\lesssim \Vert f_1\Vert_{W^{2,\sigma}} \Vert f_2\Vert_{W^{2,\sigma}}.\]
	Moreover, using the Plancherel theorem in the above estimate we get  
	\begin{eqnarray}\label{decay}
		\Vert \mathcal{A}^{u,v}_{1}(f_1\ast \psi_{i},f_2\ast\psi_j)\Vert_{L^{1}}\lesssim  2^{\sigma\max\{i,j\}}\Vert f_1\Vert_{L^{2}} \Vert f_2\Vert_{L^{2}},
	\end{eqnarray}
	where $supp(\psi_j)\subset Ann(2^{j-1},2^{j+1})$. On the other hand, the estimate $(ii)$ of Lemma \ref{mainests} implies that 
	\[\Vert\mathcal{A}^{T}_{1}(f_1\ast\psi_{i},f_2\ast\psi_j)\Vert_{L^{\frac{1}{2}}}\lesssim  2^{(i+j)s}\Vert f_1\Vert_{L^{1}} \Vert f_2\Vert_{L^{1}}.\]
	Therefore, invoking the similar machinery as the higher dimensional case $(d\geq2)$, 
	( i.e. the proof of boundedness of $\mathcal{N}^{T}_{lac}$  in the region $ \Omega(\{O,A,P,B\})$) we 
	get boundedness of $\mathcal{N}^{T}_{lac}$  for $d=1$ in the non-Banach region for small $s$. Finally, 
	applying bilinear interpolation with the Banach range estimates as above we get boundedness of 
	$\mathcal{N}^{T}_{lac}$ in the shaded region $\Omega(\{0,A,D,B\})$. The point $D$ in the 
	following figure will be determined by the value of $\sigma$ and the upper Minkowski content 
	$s$ of $T$. Note that when $T=\{0\}$, the operator $\mathcal{N}^{0}_{lac}$ is the bilinear 
	lacunary circular maximal function considered by Christ and Zhou in \cite{ChristZhou}, and the 
	$L^{p_1}\times L^{p_2}\to L^{p_3}$ boundedness of  $\mathcal{N}^{0}_{lac}$ holds for 
	$ (1/p_1,1/p_2)\in [0,1)\times [0,1)$ with $1/p_1+1/p_2=1/p_3$.

\begin{figure}[H]
	\centering
	\begin{tikzpicture}[scale=3, font=\small]
		\fill[lightgray](0,0)--(1,0)--(0.6,0.6)--(0,1)--cycle;
		\draw[densely dotted] (0,1)node[left]{$B$}--(1,1);
		\draw[densely dotted] (1,0)node[below]{$A$}--(1,1);
		\draw[densely dotted] (0,0)--(0.6,0.6)node[above]{$D$};
		\draw[densely dotted] (1,0)--(1,1);
		\draw[thin][->]  (0,0) --(1.2,0) node[right]{$\frac{1}{p_1}$};
		\draw[black,thin][->]  (0,0) --(0,1.2) node[left]{$\frac{1}{p_2}$};
		\draw[densely dotted] (0,0)node[left]{$O$}--(1,1)node[right]{$C$};
		\draw[densely dotted] (0,1)--(1,0);
	\end{tikzpicture}
	\caption{Region $\Omega \left(\{O,A,D,B\}\right)$ .}
	\label{caseB}
\end{figure}
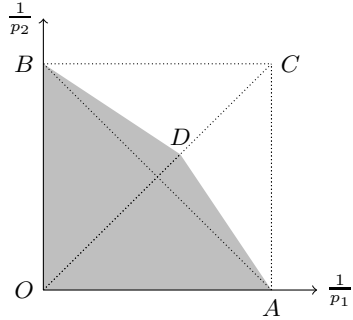

\section*{Acknowledgement}
Ankit Bhojak and Saurabh Shrivastava acknowledge the financial support 
from Science and Engineering Research Board, Department of Science and 
Technology, Govt. of India, under the scheme Core Research 
Grant, file no. CRG/2021/000230. Ankit Bhojak was also supported by the National Board of Higher Mathematics postdoctoral fellowship, file no. 0204/33/2023/R\&D-II/1141. 
Surjeet Singh Choudhary is 
supported by CSIR(NET), file no.09/1020(0182)/2019- EMR-I 
for his Ph.D. fellowship. Kalachand Shuin is supported by 
NRF grant no. 2022R1A4A1018904 and BK 21 Post doctoral fellowship and he is also supported by DST Inspire faculty fellowship with registration number IFA23-MA191. 
The authors acknowledge the support and hospitality provided by the 
International Centre for Theoretical Sciences, Bangalore (ICTS) for 
participating in the program - Modern trends in Harmonic 
Analysis (code: ICTS/Mtha2023/06).

\bibliography{biblio}
\end{document}